\documentclass[10pt]{amsart}

\usepackage{amsmath}
\usepackage{amssymb}
\usepackage{bm}
\usepackage{graphicx}
\usepackage{psfrag}
\usepackage{color}
\usepackage{hyperref}
\hypersetup{colorlinks=true, linkcolor=blue, citecolor=magenta, urlcolor=wine}
\usepackage{url}
\usepackage{algpseudocode}
\usepackage{fancyhdr}
\usepackage{mathtools}
\usepackage{tikz-cd}
\usepackage{xy}
\input xy
\xyoption{all}
\usepackage{stmaryrd}
\usepackage{calrsfs}

\newtheorem*{maintheorem*}{Main Theorem}
\newtheorem{theorem}{Theorem}[section]
\newtheorem*{theorem*}{Main Theorem}

\newtheorem{prop}[theorem]{Proposition}

\newtheorem{lemma}[theorem]{Lemma}
\newtheorem{cor}[theorem]{Corollary}
\theoremstyle{definition}
\newtheorem{definition}[theorem]{Definition}
\newtheorem{remark}[theorem]{Remark}
\newtheorem{example}[theorem]{Example}
\numberwithin{equation}{section}

\newcommand{\cc}{\mathbb{C}}

\newcommand{\nn}{\mathbb{N}}
\newcommand{\pp}{\mathbb{P}}
\newcommand{\qq}{\mathbb{Q}}
\newcommand{\rr}{\mathbb{R}}
\newcommand{\zz}{\mathbb{Z}}

\providecommand\ldb{\llbracket}
\providecommand\rdb{\rrbracket}

\keywords{join semilattice, factorization, factorizable semilattice, half-factoriality, length-factoriality, elasticity}


\begin{document}
	
\mbox{}
\title{Factoriality inside Boolean lattices}

\author{Khalid Ajran}
\address{Department of Mathematics\\MIT\\Cambridge, MA 02139}
\email{kajran@mit.edu}

\author{Felix Gotti}
\address{Department of Mathematics\\MIT\\Cambridge, MA 02139}
\email{fgotti@mit.edu}

%\date{\today}
	
\begin{abstract}
	Given a join semilattice $S$ with a minimum $\hat{0}$, the quarks (also called atoms in order theory) are the elements that cover~$\hat{0}$, and for each $x \in S \setminus \{\hat{0}\}$ a factorization (into quarks) of $x$ is a minimal set of quarks whose join is~$x$. If every element $x \in S \setminus \{\hat{0}\}$ has a factorization, then $S$ is called factorizable. If for each $x \in S \setminus \{\hat{0}\}$, any two factorizations of $x$ have equal (resp., distinct) size, then we say that~$S$ is half-factorial (resp., length-factorial). Let $B_\nn$ be the Boolean lattice consisting of all finite subsets of~$\nn$ under intersections and unions. Here we study factorizations into quarks of join subsemilattices of $B_\nn$, focused on the notions of half-factoriality and length-factoriality. We also consider the unique factorization property, which is the most special and relevant type of half-factoriality, and the elasticity, which is an arithmetic statistic that measures the deviation from half-factoriality.
\end{abstract}

\maketitle

%%%%%%%%%%%%
%%%%%%%%%%%%
\section{Introduction}
\label{sec:intro}

In a commutative monoid $M$, a non-invertible element is called irreducible provided that it does not decompose as a product of two non-invertible elements, and then $M$ is called atomic provided that every non-invertible element of $M$ decomposes as a product of finitely many irreducibles. Such irreducible decompositions are called factorizations (into irreducibles). The first systematic studies of factorizations seem to go back to the sixties with the work of Carlitz~\cite{lC60}, Narkiewicz~\cite{wN66,wN67}, and Cohn~\cite{pC68} in the context of algebraic number theory and commutative ring theory. After that, there were further sporadic investigations of factorizations, including those by Skula~\cite{lS76}, Zaks~\cite{aZ76,aZ80}, Steffan~\cite{jlS86}, and Valenza~\cite{rV90}. In the early nineties, Anderson, Anderson, and Zafrullah~\cite{AAZ90} and Halter-Koch~\cite{fHK92} introduced the bounded and the finite factorization properties. Since then, a flurry of papers studying factorizations in a large variety of algebraic structures have systematically appeared in the literature, giving shape to what we call today factorization theory (see the current surveys~\cite{AG22,GZ20} and references therein).
\smallskip

When the commutative monoid is taken to be a join semilattice $S$ with a minimum, which must be the identity and only invertible element of $S$, the notion of a factorization is vacuous because~$S$ contains no irreducibles (as every element of $S$ is idempotent). However, a rich factorization theory in join semilattices, parallel to that of factorizations into irreducibles, is still possible if we let the role of irreducibles be played by the quarks. Following Tringali~\cite{sT22}, we say that a quark\footnote{Although quarks have been called atoms in order theory for quite a while, we have adopted the former term here because, in the context of factorization theory, the latter term is reserved for a special and well-studied type of irreducible elements.} of a join semilattice is an element covering the minimum. For the rest of this section, let $S$ be a join semilattice with a minimum, which we denote by~$\hat{0}$. Then we say that join semilattice $S$ is factorizable (into quarks) provided that every element in $S \setminus \{\hat{0}\}$ is the join of finitely many quarks. For $x \in S \setminus \{\hat{0}\}$, a factorization (into quarks) of $x$ in~$S$ is a minimal set of quarks whose join is~$x$. Clearly, $S$ is factorizable if and only if every element in $S \setminus \{\hat{0}\}$ has a factorization. In this paper, we mostly focus on the study of factorizations in the class consisting of factorizable join subsemilattices of the free semilattice on~$\nn$, the set of positive integers. Therefore we will be investigating factorizations in a highly non-cancellative setting: indeed, in a join semilattice with a minimum, the only cancellative element is the minimum (as every element is idempotent).
\smallskip

Having said that, it is worth emphasizing that factorizations (into irreducibles) have been well studied in various non-cancellative algebraic structures during the past few decades (for instance, see \cite{CFGO16} and references therein). Factorizations in commutative rings with zero-divisors were studied by Anderson and Valdes-Leon~\cite{AV96, AV97} in the nineties and more recently by Anderson and Chun~\cite{AC11,AC13}. In the same direction, Juett et al. have provided a further insight into the same area with their recent papers~\cite{AJ17,EJ21,JM22}. The arithmetic of factorizations of semigroups of ideals and modules was studied by Fan et al. in~\cite{FGKT17}. On the other hand, factorizations in power monoids have also been the subject of recent investigation. Given a commutative monoid $M$, the set consisting of all nonempty finite subsets of $M$ under the Minkowski sum is called the power monoid of~$M$. Motivated in part by arithmetic combinatorics, factorizations in the power monoid of the additive monoid of nonnegative integers were first studied by Fan and Tringali in~\cite{FT18}, where the term ``power monoid" was coined. The same were later studied in power monoids of cyclic groups by Antonious and Tringali~\cite{AT21} as well as in power monoids of numerical monoids by Bienvenu and Geroldinger~\cite{BG23}. Finally, factorizations in a more general setting have been studied by Cossu and Tringali in a series of recent papers (see~\cite{CT22} and references therein).
\smallskip

Let $S$ be a factorizable join semilattice with minimum $\hat{0}$. For each $x \in S \setminus \{ \hat{0}\}$, the size of a factorization of~$x$ is called a length of $x$, and the set consisting of all the possible lengths of~$x$ is denoted by $\mathsf{L}(x)$ and called the set of lengths of $x$. We say that $S$ is half-factorial (resp., length-factorial) if any two distinct factorizations of the same non-minimum element of~$S$ have the same length (resp., different lengths). Let $B_\nn$ be the free lattice on $\nn$. As pointed out earlier, our primary purpose here is to investigate factorizations inside join subsemilattices of $B_\nn$. We call such join subsemilattices Boolean sublattices. An effective traditional way to study the phenomenon of multiple factorizations (into irreducibles) is through the lens of sets of factorizations/lengths as, intuitively, the same phenomenon is more pronounced when the sizes of the sets of factorizations/lengths are larger. The purpose of this paper is twofold. On one hand, we collect fundamental evidence to conclude that factorizations in Boolean sublattices are rather nontrivial and quite interesting; we not only identify classes of half/length-factorial Boolean sublattices, but also show that there exist Boolean sublattices that are as far from being half-factorial as they could possibly be (we do this by considering elasticities). On the other hand, we compare and contrast some of the results we establish here with results already known in the more classical setting of factorizations into irreducibles.
\smallskip

The term ``half-factoriality" was coined by Zaks in~\cite{aZ76}. With initial motivation in algebraic number theory, the notion of half-factoriality has been systematically investigated in the setting of monoids and domains during the last six decades (see the classical works~\cite{lC60,jC97,lS76,aZ80} and the more recent papers~\cite{GZ19,PS20}). The unique factorization property is the most special case of half-factoriality, and the study of the former in the setting of monoids has been largely motivated by trying to establish results similar to the Fundamental Theorem of Arithmetic in more abstract algebraic structures. For instance, although Dedekind domains (specially, rings of integers) are not UFDs in general, their nonzero ideals factor uniquely into prime ideals. This important factorization property of Dedekind domains has inspired many authors to investigate further classes of commutative monoids satisfying the unique factorization property (see \cite{rJ71} and the more recent paper \cite{SH08}). We devote Section~\ref{sec:half-factoriality} to the study of half-factoriality in Boolean sublattices. In Subsection~\ref{subsec:factoriality}, we consider the special case of the unique factorization property: we find a sufficient condition for a Boolean sublattice to be a unique factorization semilattice and then we identify a special class of unique factorization Boolean sublattices. In Subsection~\ref{subsec:half-factoriality}, we establish two characterizations of half-factoriality in the class consisting of all Boolean sublattices whose quarks have size at most~$2$ (a similar characterization is given in Subsection~\ref{subsec:factoriality} for the unique factorization property).
\smallskip

In the classical theory of factorizations into irreducibles, one of the most important arithmetic statistics to measure the deviation of an algebraic structure from being half-factorial is the elasticity. A similar notion of elasticity can be defined in our setting of join semilattices by replacing factorizations into irreducibles by factorizations into quarks. For $x \in S \setminus \{\hat{0}\}$, we call $\sup \mathsf{L}(x)/\min \mathsf{L}(x)$ the elasticity of~$x$ and we denote it by $\rho(x)$. The elasticity of the whole semilattice $S$ is then defined as $\sup \{\rho(x) \mid x \in S \setminus \{ \hat{0}\} \}$. The notion of elasticity was first considered in the context of rings of algebraic integers and Dedekind domains by Valenza~\cite{rV90} and Steffan~\cite{jlS86}, respectively (the term ``elasticity" being coined by Valenza in~\cite{rV90}). The elasticity has been systematically studied ever since not only in integral domains, but also in the more general setting of cancellative commutative monoids: see the Anderson's survey~\cite{dfA97} for the most relevant advances on the elasticity in integral domains until 2000 and see the papers \cite{BC16,fG20,fK05,qZ19} (and references therein) for further advances on the elasticity in both commutative monoids and integral domains taking place during the past two decades. Even more recently, the elasticity has also been considered for general monoids (see~\cite{CT23}). In Section~\ref{sec:elasticity}, we establish a realization theorem for the elasticity of such join semilattices: we prove that if $\alpha$ is a real number that is at least~$1$, then there exists a factorizable Boolean sublattice whose elasticity is~$\alpha$ (we argue the same statement for $\alpha = \infty$). This result is parallel to~\cite[Theorem~3.2]{AA92} and~\cite[Proposition~3.5]{CHM06} and in sharp contrast to \cite[Theorems~5.7 and~5.8]{fG20}.
\smallskip

The notion of length-factoriality was introduced by Coykendall and Smith in~\cite{CS11} to characterize UFDs: they proved that an integral domain is a UFD if and only if its multiplicative monoid is length-factorial. Observe that the notion of length-factoriality somehow complements that of half-factoriality in the sense that a monoid has the unique factorization property if and only if it is half-factorial and length-factorial simultaneously. This justifies why length-factoriality was initially introduced in~\cite{CS11} under the term ``other-half-factoriality" (the term ``length-factoriality" was adopted in~\cite{CCGS21} recently). Length-factoriality was first investigated in the context of commutative monoids in \cite[Section~5]{fG20a}, \cite[Sections~3 and~5]{CCGS21}, and \cite[Section~5]{CG22}. Length-factoriality was also studied in~\cite{GZ21} in the setting of Krull monoids. In addition, length-factoriality has been recently investigated in~\cite[Section~5]{BVZ23} and \cite[Section~6]{GP23} in the setting of semidomains. In Section~\ref{sec:length-factoriality}, we study length-factoriality in Boolean sublattices. For Boolean sublattices whose quarks have size at most~$2$, we prove that the properties of being length-factorial and that of having unique factorization are equivalent. This result is parallel to the main result of~\cite{CS11} for integral domains. In the same section, we construct a factorizable join semilattice~$S$ that is length-factorial but contains one element whose set of lengths is not finite. This construction is not possible in the setting of cancellative commutative monoids (under the most classical notion of factorizations into irreducibles) as if a cancellative commutative monoid is length-factorial, then the set of factorizations of any non-invertible element must be finite~\cite[Proposition~3.1]{BVZ23}.

\bigskip
%%%%%%%%%%%%
%%%%%%%%%%%%
\section{Preliminary}
\label{sec:prelim}

In this section we briefly review some notation and terminology we shall be using throughout this paper. For undefined terms in posets and lattices, see \cite[Chapter~3]{rS12} by Stanley and for a comprehensive treatment in factorization theory, see the manuscript~\cite{GH06} by Geroldinger and Halter-Koch.

\smallskip
%%%%%%%%%%%%%%%%
\subsection{General Notation}

As it is customary, $\zz$, $\qq$, $\rr$, and $\cc$ will denote the set of integers, rational numbers, real numbers, and complex numbers, respectively. We let $\nn$ and $\nn_0$ denote the set of positive and nonnegative integers, respectively. In addition, we let $\pp$ denote the set of primes. For $S \subseteq \rr$ and $r \in \rr$, we set $S_{\ge r} = \{s \in S \mid s \ge r\}$ and we use the notations $S_{> r}, S_{\le r}$, and $S_{< r}$ in a similar manner. For a set $S$, we let $2^S$ denote the power set of $S$; that is, $2^S = \{X \mid X \subseteq S\}$. When two sets~$S$ and~$T$ are disjoint, we often write $S \sqcup T$ instead of $S \cup T$ to emphasize that we are taking the union of disjoint sets.

\medskip
%%%%%%%%%%%
\subsection{Posets}

Let $P$ be a poset with its order relation denoted by $\le$. A subset~$I$ of $P$ is called an \emph{order ideal} if for all $x \in P$ and $y \in I$, the relation $x \le y$ implies that $x \in I$. It is clear that for each $y \in P$, the set $\Lambda_y := \{x \in P \mid x \le y\}$ is an order ideal. For $x,y \in P$, we set $[x,y] := \{ t \in P \mid x \le t \le y\}$ and call the subposet $[x,y]$ of~$P$ the \emph{interval} from $x$ to $y$ (clearly, $[x,y]$ is empty if and only if $x \not\leq y$). We say that~$P$ is \emph{locally finite} if $[x,y]$ is finite for all $x,y \in P$. When the poset~$P$ is the set of integers with the standard order, we write $\ldb x,y \rdb$ instead of $[x,y]$. For $x,y \in P$, we say that~$y$ \emph{covers} $x$ and write $x \lessdot y$ provided that $|[x,y]| = 2$. A \emph{chain} of $P$ is a finite sequence of elements $x_0, x_1, \dots, x_\ell \in P$ such that $x_0 < x_1 < \dots < x_\ell$. The \emph{length} of a chain $x_0, x_1, \dots, x_\ell$ is $\ell$, and we say that the chain is \emph{saturated} if $x_0 \lessdot x_1 \lessdot \cdots \lessdot x_\ell$. We say that~$P$ is \emph{graded} if there exists a \emph{rank function} on $P$, that is, a function $\rho \colon P \to \nn_0$ satisfying that $\rho(m) = 0$ for every minimal element $m \in P$ and $\rho(y) = \rho(x)+1$ for all $x,y \in P$ such that $x \lessdot y$. Observe that if a poset is graded, then it has a unique rank function. Furthermore, one can readily che\underline{}ck that the poset~$P$ is graded if and only if for all $x,y \in P$ the interval $[x,y]$ is finite and any two saturated chains from $x$ to $y$ have the same length.

\medskip
%%%%%%%%%%%%%%
\subsection{Semilattices}

If $P$ contains a minimum, such a minimum is unique and we denote it by $\hat{0}$. If a least upper bound of a finite subset $X$ of $P$ exists, then it must be unique and it is called the \emph{join} of $X$ in $P$. In such a case, we let $\vee X$ denote the join of $X$. A poset containing a minimum and satisfying that each finite subset has a join is called a \emph{join semilattice}. The dual notion of a join semilattice is that of a meet semilattice: $P$ is called a \emph{meet semilattice} if it contains a maximum, which must be unique and is denoted by $\hat{1}$, and each finite subset~$X$ of $P$ has a meet (i.e., a greatest lower bound), which must be unique and is denoted by~$\wedge X$. For $x_1, \dots, x_n \in P$, if the join (resp., meet) of $\{x_1, \dots, x_n\}$ exists, then we denote it by $x_1 \vee \dots \vee x_n$ (resp., $x_1 \wedge \dots \wedge x_n$). A \emph{lattice} is a poset that is simultaneously a join and a meet semilattice. Since the notion of a meet semilattice does not play a fundamental role in the context of this paper, we will reserve the single term ``semilattice" to refer to a join semilattice. One can readily verify that every locally finite semilattice~$S$ is a lattice, where $\wedge X = \vee \{s \in S \mid s \le x \text{ for all } x \in X\}$ for any finite subset $X$ of $S$. For this reason, all the semilattices we consider in this paper, except the one constructed in Example~\ref{ex:LFL not FFL}, are indeed lattices. 
%\smallskip 
Let $S$ be a semilattice. For a subset~$X$ of $S$, the semilattice \emph{generated} by $X$, which we denote here by $\langle X \rangle$, is the smallest subsemilattice of $S$ containing the set $\{\hat{0}\} \cup X$. If $x_1, \dots, x_n \in S$, then we denote the subsemilattice of $S$ generated by $\{x_1, \dots, x_n\}$ simply as $\langle x_1, \dots, x_n \rangle$. An element $a \in S$ is called a \emph{quark} provided that $\hat{0} \lessdot a$. We let $\mathcal{A}(S)$ denote the set of quarks of~$S$. We say that~$S$ is \emph{factorizable into quarks} or simply \emph{factorizable} if it is generated by its set of quarks; that is, $S = \langle \mathcal{A}(S) \rangle$. A lattice $S$ is \emph{semimodular} if it is graded and its rank function $\rho$ satisfies the following property: $\rho(x) + \rho(y) \ge \rho(x \wedge y) + \rho(x \vee y)$ for all $x,y \in S$. Factorizable semimodular lattices are called \emph{geometric lattices}.

\medskip
%%%%%%%%%%%%%%%%%%
\subsection{Boolean Sublattices}

Let $Y$ be a set. The collection $B_Y$ consisting of all finite subsets of $Y$ is a poset under inclusion of sets. The poset $B_Y$ is a lattice with joins and meets given by unions and intersections, respectively. The lattice $B_Y$ is often called the \emph{Boolean lattice} on $Y$. It is clear that $B_Y$ is a geometric lattice whose quarks are the singletons. The subsemilattices of $B_\nn$ are the primary focus of this paper, and we call them \emph{Boolean sublattices}. Since $B_\nn$ is locally finite, every Boolean sublattice is a locally finite semilattice and, therefore, a lattice (although meets may not be given by intersections). If \emph{all} the elements of a Boolean sublattice are sets comprising only $1$-digit positive integers (i.e., integers in the interval $\ldb 1,9 \rdb$), then in order to simplify notation we omit the braces and the commas in the set representation of its elements: for instance, instead of $\langle \{1,2\}, \{2,3\}, \{1,3\} \rangle$, we write $\langle 12, 23, 13 \rangle$. Unlike the lattice $B_\nn$, Boolean sublattices are not necessarily graded or factorizables. For instance, one can easily see that the Boolean sublattice $\langle \{1\}, \{2\}, \{2,3\}, \ldb 1,n \rdb \mid n \ge 3 \rangle$ is neither graded nor factorizable. Given a collection $A$ of finite subsets of~$\nn$, the Boolean sublattice $\langle A \rangle$ generated by $A$ is a lattice and, if there are no inclusion relations between any two members of $A$, then it is clear that $\langle A \rangle$ is a factorizable lattice with set of quarks $A$.

\medskip
%%%%%%%%%%%
\subsection{Graphs}

Throughout this paper, graphs are tacitly assumed to be undirected and simple (i.e., without loops and multiple edges), but they are allowed to have infinitely many vertices and edges. Let~$G$ be a graph. We let $V(G)$ and $E(G)$ denote the set of vertices and edges of $G$, respectively. We say that a sequence $\{v_0, v_1\}, \{v_1, v_2\}, \dots, \{v_{\ell-1}, v_\ell \} \in E(G)$ is a \emph{path} if the vertices $v_0, \dots, v_\ell$ are distinct. On the other hand, a sequence $\{v_0, v_1\}, \{v_1, v_2\}, \dots, \{v_{\ell-1}, v_\ell \} \in E(G)$ is a \emph{cycle} if $\{v_0, v_1\}, \{v_1, v_2\}, \dots, \{v_{\ell-2}, v_{\ell-1} \}$ is a path and $v_\ell = v_0$. The number of edges in a path/cycle is called its \emph{length}. If a graph is connected and all its edges form a path (resp., a cycle), then it is called a \emph{path graph} (resp., \emph{cycle graph}) and is denoted by $P_n$ (resp., $C_n$), where $n$ is the number of vertices. The \emph{diameter} of a tree $T$ is the length of a maximum-length path in~$T$. The \emph{star graph} on $n+1$ vertices, denoted by $S_n$, is the complete bipartite graph $K_{n,1}$, that is, a tree with $n+1$ vertices and diameter at most~$2$. A subset $E$ of $E(G)$ is called an \emph{edge covering} of the graph $G$ provided that every vertex of $G$ is one of the two incidence vertices of an edge in $E$.

\medskip
%%%%%%%%%%%%%%%
\subsection{Factorizations}

For a set $X$, the \emph{free semilattice} on $X$, denoted here by $F(X)$, is the collection of all nonempty finite subsets of~$X$ ordered by inclusion. Let $S$ be a semilattice. The \emph{factorization semilattice} of $S$ is the free semilattice $F(\mathcal{A}(S))$ on $\mathcal{A}(S)$, and we denote it by $\mathsf{Z}(S)$. Since $\mathsf{Z}(S)$ is free, there exists a unique semilattice homomorphism $\pi \colon \mathsf{Z}(S) \to S$ satisfying that $\pi(a) = a$ for all $a \in \mathcal{A}(S)$ (observe that $S$ is factorizable if and only if $\pi$ is surjective). For $a_1, \dots, a_\ell \in \mathcal{A}(S)$, we say that $z :=\{a_1, \dots, a_\ell\}$ is a \emph{factorization into quarks} or simply a \emph{factorization} of $a_1 \vee \dots \vee a_\ell$ provided that the formal join $a_1 \vee \dots \vee a_\ell$ is irredundant; that is, if $\vee \{a_j \mid j \in J\} < a_1 \vee \dots \vee a_\ell$ in $S$ whenever $J \subsetneq \ldb 1,\ell \rdb$. For each $z \in \mathsf{Z}$, we call $|z|$ the \emph{length} of the factorization~$z$. For each $x \in S$, we let $\mathsf{Z}(x)$ denote the set of all factorizations of~$x$, and we set
\[
	\mathsf{L}(x) := \{|z| : z \in \mathsf{Z}(x)\}.
\]
We say that $S$ is a \emph{finite factorization semilattice} (or an \emph{FFS}) if $|\mathsf{Z}(x)| < \infty$ for all $x \in S$. When $\mathsf{Z}(x)$ is a singleton for every $x \in S \setminus \{\hat{0}\}$, we say that $S$ is a \emph{unique factorization semilattice} (or a \emph{UFS}). It follows from the definitions that every UFS is an FFS. When $\mathsf{L}(x)$ is a singleton for all $x \in S \setminus \{\hat{0}\}$, we say that $S$ is a \emph{half-factorial semilattice} (or an \emph{HFS}). On the other hand, we say that~$S$ is a \emph{length-factorial semilattice} (or an \emph{LFS}) provided that any two distinct factorizations of the same non-minimum element of~$S$ have different lengths. It follows from the definitions that a semilattice is a UFS if and only if it is both an HFS and an LFS.
\smallskip

Let $S$ be a factorizable Boolean sublattice. For any $X \in S$, the set $\mathcal{A}(S) \cap \Lambda_X$ is finite; indeed, $|\mathcal{A}(S) \cap \Lambda_X| \le 2^{|X|}$, which implies that $\mathsf{Z}(X)$ is finite. Therefore being factorizable and being an FFS are equivalent conditions in the class consisting of Boolean sublattices. We record this observation for future reference.

\begin{remark} \label{rem:atomic = FF in Boolean sublattices}
	A Boolean sublattice is factorizable if and only if it is an FFS.
\end{remark}
\smallskip

The corresponding statement for the more general class consisting of all semilattices does not hold. In Section~\ref{sec:length-factoriality}, we will construct a factorizable join semilattice that is not an FFS.

\bigskip
%%%%%%%%%%%%%%%%
%%%%%%%%%%%%%%%%
\section{Two Related Graphs}
\label{sec:half-factoriality}

In this section, we introduce two graphs that will help us investigate factorizations into quarks inside Boolean sublattices.

\begin{definition}
	The \emph{quarkic graph} of a Boolean sublattice $S$, denoted by $\mathcal{G}(S)$, is the graph whose set of vertices is $\mathcal{A}(S)$ and that has an edge between distinct quarks $A$ and $B$ whenever $A \cap B$ is nonempty.
\end{definition}

We say that a quark $A$ of a Boolean sublattice $S$ is an \emph{isolated quark} if $A$ is disjoint from any other quark of $S$. We let $\mathcal{A}_I(S)$ denote the set consisting of all the isolated quarks of $S$.

\begin{example}
	Consider the Boolean sublattice $S := \langle 12, 13, 23, 45, 46 \rangle$. A fragment of the Hasse diagram of $S$ is illustrated in Figure~\ref{fig:atomic graph}. It is clear that $S$ is factorizable with $\mathcal{A}(S) = \{ 12, 13, 23, 45, 46 \}$. Observe that the quarkic graph $\mathcal{G}(S)$ of $S$ consists of two connected components: the $3$-cycle graph $G_1$ on the set of vertices $\{12, 13, 23\}$ and the path graph $G_2$ on the set of vertices $\{45,46\}$. Therefore $S$ contains no isolated quarks. In Figure~\ref{fig:atomic graph}, the Boolean sublattices $\langle V(G_1) \rangle$ and $\langle V(G_2) \rangle$ are highlighted in red and green, respectively.
	\begin{figure}[h]
		\includegraphics[width=8cm]{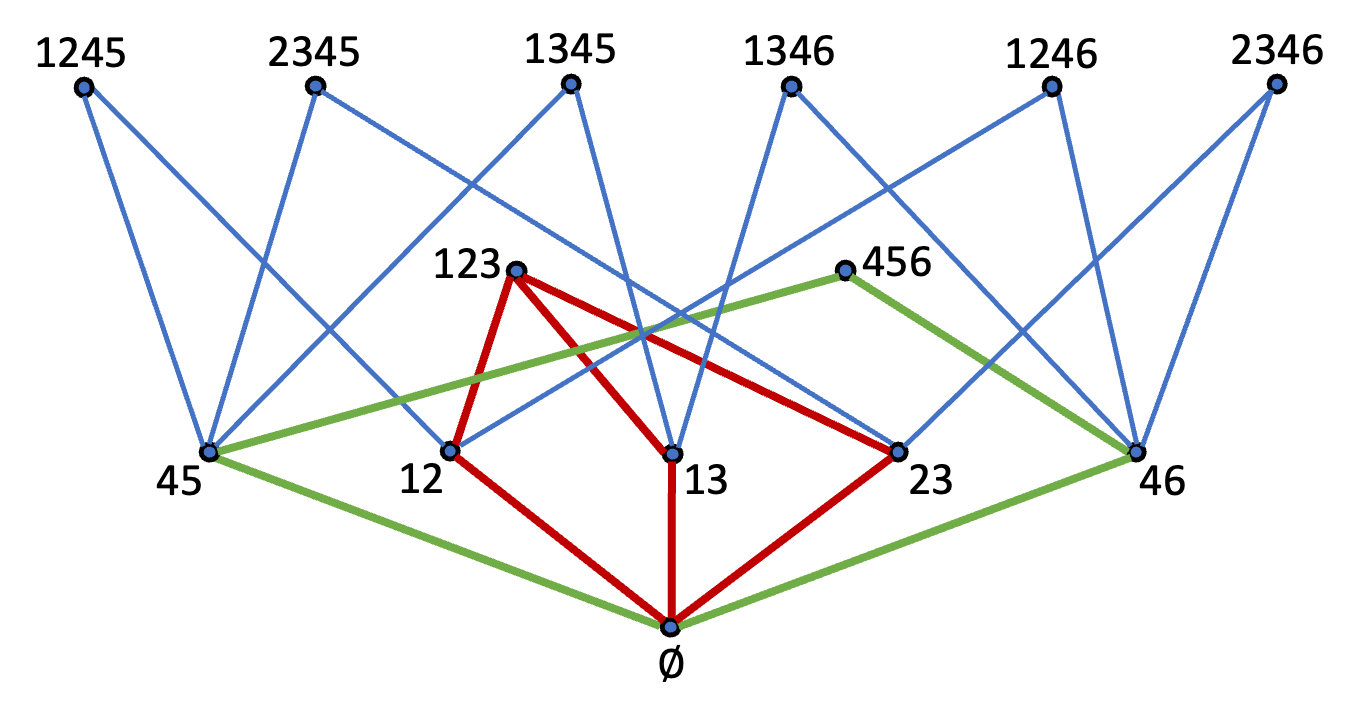}
		\caption{A fragment of the Hasse diagram of the Boolean sublattice generated by the set $\{ 123, 234, 345, 456 \}$ (including only the sets of size at most $4$).}
		\label{fig:atomic graph}
	\end{figure}
\end{example}

\begin{lemma} \label{lem:decomposition lemma}
	For a factorizable Boolean sublattice $S$, the following statements hold.
	\begin{enumerate}
		\item If $C$ is a connected component of $\mathcal{G}(S)$, then the Boolean sublattice $\langle V(C) \rangle$ is factorizable and $\mathcal{A}(\langle V(C) \rangle) = \mathcal{A}(S) \cap \langle V(C) \rangle$.
		\smallskip
		
		\item For each nonempty set $X \in S$, there exist unique connected components $C_1, \dots, C_k$ of $\mathcal{G}(S)$ and unique nonempty sets $X_1 \in \langle V(C_1) \rangle, \dots, X_k \in \langle V(C_k) \rangle$ such that
		\begin{equation*} \label{eq:canonical decomposition}
			X = X_1 \sqcup \dots \sqcup X_k.
		\end{equation*}
		%\smallskip
		
		\item With notation as in part~(2), the function $\mathsf{Z}_S(X) \to \mathsf{Z}_{\langle V(C_1) \rangle}(X_1) \times \dots \times \mathsf{Z}_{\langle V(C_k) \rangle}(X_k)$ given by the assignment $z \mapsto (z_1, \dots, z_k)$, where $z_i := \{A \in z \mid A \le X_i\}$, is a bijection.
		\smallskip
		
		\item For all $X \in S$ and $A \in \mathcal{A}_I(S)$ with $A \not\le X$,
		\[
			\mathsf{Z}(X \vee A) = \{z \cup \{A\} \mid z \in \mathsf{Z}(X) \}.
		\]
	\end{enumerate}
\end{lemma}

\begin{proof}
	(1) Let $C$ be a connected component of $\mathcal{G}(S)$. As $\langle V(C) \rangle$ is a subposet of $S$ containing the minimum of $S$, the inclusion $\mathcal{A}(S) \cap \langle V(C) \rangle \subseteq \mathcal{A}(\langle V(C) \rangle)$ holds. On the other hand, it is clear that $A \not\leq A'$ for any $A \in \mathcal{A}(S) \setminus V(C)$ and $A' \in V(C)$, which implies that $\mathcal{A}(\langle V(C) \rangle) \subseteq \mathcal{A}(S) \cap \langle V(C) \rangle$. Hence $\mathcal{A}(\langle V(C) \rangle) = \mathcal{A}(S) \cap \langle V(C) \rangle$. Verifying that $\langle V(C) \rangle$ is factorizable amounts to observing that, for each nonempty set $X \in \langle V(C) \rangle$, the fact that $A \not\leq X$ for any $A \in \mathcal{A}(S) \setminus V(C)$ guarantees that each factorization of $X$ in~$S$ is also a factorization of $X$ in $\langle V(C) \rangle$.
	\smallskip
	
	(2) Fix a nonempty set $X \in S$. Because $X$ is a finite set, there are only finitely many connected components $C$ of $\mathcal{G}(S)$ such that $X$ intersects some quarks of $C$. Let $C_1, \dots, C_k$ be such connected components. For each $i \in \ldb 1,k \rdb$, set $X_i = \bigcup \{A \in V(C_i) \mid A \subseteq X\}$. %$X_i = \bigcup_{A \in C_i} X \cap A$. 
	It is clear that each $X_i$ is nonempty. Also, $X = X_1 \sqcup \dots \sqcup X_k$ because quarks of $S$ in distinct connected components of $\mathcal{G}(S)$ are disjoint. The uniqueness of the decomposition follows immediately.
	\smallskip
	
	(3) Let $\varphi_X \colon z \mapsto (z_1, \dots, z_k)$ be such a function. Fix a factorization $z \in \mathsf{Z}_S(X)$, and observe that for each $i \in \ldb 1,k \rdb$, the equality $\vee \{A \in z \mid A \le X_i\} = X_i$. This, in tandem with the irredundance of~$z$, ensures that $z_i \in \mathsf{Z}_{\langle V(C_i) \rangle}(X_i)$. Thus, $\varphi_X$ is well defined. On the other hand, one can readily check that the function $\mathsf{Z}_{\langle V(C_1) \rangle}(X_1) \times \dots \times \mathsf{Z}_{\langle V(C_k) \rangle}(X_k) \to \mathsf{Z}_S(X)$ given by $(z_1, \dots, z_k) \mapsto z_1 \sqcup \dots \sqcup z_k$ is the inverse of $\varphi_X$. Hence $\varphi_X$ is a bijection.
	\smallskip
	
	(4) This is an immediate consequence of part~(3).
\end{proof}

\begin{remark}
	A Boolean sublattice $S$ may not be factorizable even when $\langle V(C) \rangle$ is factorizable for each connected component~$C$ of $\mathcal{G}(S)$; for instance, although the Boolean sublattice $S := \{\emptyset, 1, 12\}$ is not factorizable, it has only one connected component, namely $C_1$ with $V(C_1) = \{1\}$ and $E(C_1) = \emptyset$, and the Boolean sublattice $\langle V(C_1) \rangle = \{\emptyset, 1\}$ is clearly factorizable.
\end{remark}

Boolean sublattices whose quarks have size at most $2$ play an important role in this paper. There is a natural way to construct a graph from any factorizable Boolean sublattice whose quarks have size at most~$2$.

\begin{definition}
	Let $S$ be a factorizable Boolean sublattice whose quarks have size at most~$2$. The \emph{pairing graph} of $S$, denoted by $\mathcal{G}_p(S)$, is the graph whose set of vertices is $\nn$ and having an edge between distinct vertices $a$ and $b$ whenever $\{a,b\} \in \mathcal{A}(S)$.
\end{definition}

\bigskip
%%%%%%%%%%%%%%%%%%%%%%%%
%%%%%%%%%%%%%%%%%%%%%%%%
\section{The Unique Factorization Property}
\label{subsec:factoriality}

The main purpose of this section is to study the unique factorization property in Boolean sublattices. As the following proposition indicates, the factoriality of a factorizable Boolean sublattice determines and is fully determined by that of the Boolean sublattices corresponding to the connected components of its quarkic graph.

\begin{prop} \label{prop:factoriality and connected components}
	Let $S$ be a factorizable Boolean sublattice, and let $\mathcal{C}$ be the set consisting of all the connected components of $\mathcal{G}(S)$. Then $S$ is a UFS if and only if $\langle V(C) \rangle$ is a UFS for all $C \in \mathcal{C}$.
\end{prop}

\begin{proof}
	For the direct implication, suppose that $S$ is a UFS. Let $C$ be a connected component of $\mathcal{G}(S)$. For each nonempty set $X \in \langle V(C) \rangle$, the inclusion $\mathcal{A}(\langle V(C) \rangle) \subseteq \mathcal{A}(S) \cap \langle V(C) \rangle$ guarantees that $\mathsf{Z}_{\langle V(C) \rangle}(X) \subseteq \mathsf{Z}_S(X)$, and so $X$ has at most one factorization in $\langle V(C) \rangle$. Since $\langle V(C) \rangle$ is factorizable by part~(1) of Lemma~\ref{lem:decomposition lemma}, it must be a UFS.
	
	Conversely, suppose that $\langle V(C) \rangle$ is a UFS for all $C \in \mathcal{C}$. Fix a nonempty set $X \in S$. By part~(2) of Lemma~\ref{lem:decomposition lemma}, we can uniquely decompose $X$ as $X = X_1 \sqcup \dots \sqcup X_k$ with $X_i \subseteq \langle V(C_i) \rangle$ for every $i \in \ldb 1,k \rdb$, where $C_1, \dots, C_k$ are distinct connected components of $\mathcal{G}(S)$. Now it follows from part~(3) of Lemma~\ref{lem:decomposition lemma} that $|\mathsf{Z}_S(X)| = |\mathsf{Z}_{\langle V(C_1) \rangle}(X_1) \times \dots \times \mathsf{Z}_{\langle V(C_k) \rangle}(X_k)|$. As $\langle V(C_1) \rangle, \dots, \langle V(C_k) \rangle$ are UFSs, $\mathsf{Z}_S(X)$ must be a singleton. Then we conclude that $S$ is a UFS.
\end{proof}

The following corollary, which we will use later, is an immediate consequence of Proposition~\ref{prop:factoriality and connected components}

\begin{cor} \label{cor:isolated atoms UFL}
	A Boolean sublattice $S$ is a UFS if and only if $\langle \mathcal{A}(S) \setminus \mathcal{A}_I(S) \rangle$ is a UFS.
\end{cor}

For a Boolean sublattice $S$, we say that a quark $A$ of $S$ is an \emph{excess quark} if for each $j \in A$ there exists $B \in \mathcal{A}(S) \setminus \{A\}$ such that $j \in B$; that is, $A$ is contained in the union of the remaining quarks. We can use excess quarks to provide a sufficient condition for a Boolean sublattice to be a UFS.

\begin{prop} \label{prop:UF sufficient condition}
	Let $S$ be a Boolean sublattice. If the excess quarks of $S$ are pairwise disjoint, then~$S$ is a UFS.
\end{prop}

\begin{proof}
	Assume that no two distinct excess quarks of $S$ intersect. Suppose, towards a contradiction, that a nonempty element of $S$ has two different factorizations, namely, $\alpha$ and $\alpha'$. Then we can write $\alpha = \beta \sqcup \gamma$ and $\alpha' = \beta' \sqcup \gamma$ for factorizations $\beta, \beta', \gamma \in \mathsf{Z}(S)$ such that $\beta$ and $\beta'$ are disjoint. The fact that $\alpha$ and $\alpha'$ are distinct factorizations of the same element guarantees that both $\beta$ and $\beta'$ are nonempty ($\gamma$ may be empty). Therefore, as both~$\alpha$ and $\alpha'$ are factorizations in $S$, both sets $(\vee \beta) \setminus (\vee \gamma)$ and $(\vee \beta') \setminus (\vee \gamma)$ are nonempty. We claim that $(\vee \beta) \setminus (\vee \gamma) = (\vee \beta') \setminus (\vee \gamma)$. To argue this, take $x \in (\vee \beta) \setminus (\vee \gamma)$. Then $x \in \vee \beta \subseteq \vee (\beta \sqcup \gamma) = \vee (\beta' \sqcup \gamma)$, and so there exists $A \in \beta' \sqcup \gamma$ such that $x \in A$. As $x \notin \vee \gamma$, we see that $A \in \beta'$ and so $x \in \vee \beta'$. Hence $x \in (\vee \beta') \setminus (\vee \gamma)$. As a consequence, the inclusion $(\vee \beta) \setminus (\vee \gamma) \subseteq (\vee \beta') \setminus (\vee \gamma)$ holds. The reverse inclusion can be argued similarly. 
	
	We proceed to show that every quark in $\beta \sqcup \beta'$ is an excess quark. Take $A \in \beta \cup \beta'$. Assume first that $A \in \beta$. Fix $x \in A$. If $x \in \vee \gamma$, then there exists a quark in $\gamma$ containing $x$. Suppose, on the other hand, that $x \notin \vee \gamma$. Then $x \in (\vee \beta) \setminus (\vee \gamma) = (\vee \beta') \setminus (\vee \gamma)$, which implies that $x \in \vee \beta'$. Therefore there exists a quark $B \in \beta'$ such that $x \in B$. Thus, every element of $A$ belongs to some of the quarks in either $\gamma$ or $\beta'$. Since $A \in \beta$, and $\beta$ is disjoint from $\beta' \sqcup \gamma$, we conclude that $A$ is an excess quark. We can similarly arrive to the same conclusion under the assumption that $A \in \beta'$. As a result, every quark in $\beta \sqcup \beta'$ is an excess quark.
	
	Since $(\vee \beta) \setminus (\vee \gamma) = (\vee \beta') \setminus (\vee \gamma)$, every element in $x \in (\vee \beta) \setminus (\vee \gamma)$ belongs to the intersection of a quark $B \in \beta$ and a quark $B' \in \beta'$, which are both excess quarks. Since $\beta$ and $\beta'$ are disjoint and $x \in B \cap B'$, we see that $B$ and $B'$ are overlapping distinct excess quarks of $S$, which is a contradiction.
\end{proof}

As the following example shows, the converse of Proposition~\ref{prop:UF sufficient condition} does not hold.

\begin{example}
	Consider the Boolean sublattice $S := \langle 123, 234, 345, 456 \rangle$, whose Hasse diagram is illustrated in Figure~\ref{fig:UFS with intersecting excess atoms}. It is clear that $\mathcal{A}(S) = \{ 123, 234, 345, 456 \}$ and also that both $234$ and $345$ are excess quarks with nontrivial intersection. Still, one can readily verify that $S$ is a UFS.
	\begin{figure}[h]
		\includegraphics[width=6cm]{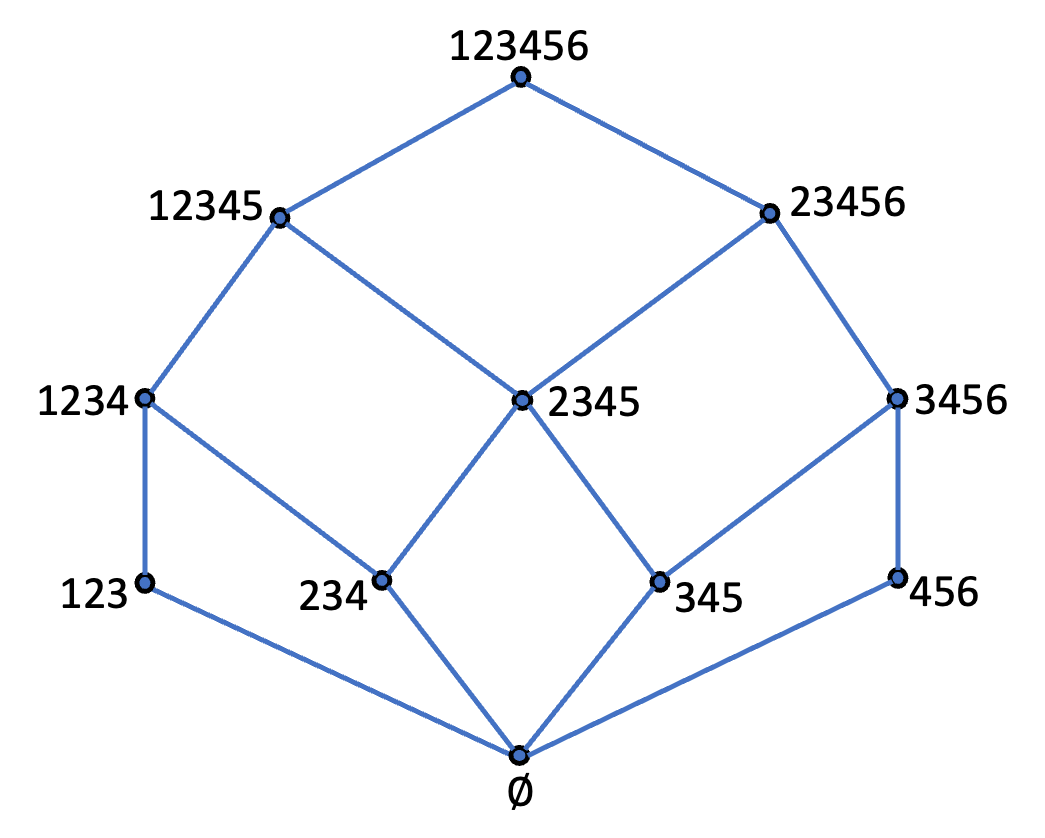}
		\caption{Hasse diagram of the Boolean sublattice generated by the set $\{ 123, 234, 345, 456 \}$.}
		\label{fig:UFS with intersecting excess atoms}
	\end{figure}
\end{example}

Our next goal is to characterize the Boolean sublattices whose quarks have size at most~$2$ that are UFS: we do this based on the corresponding pairing graphs. First, we need the following lemma.

\begin{lemma} \label{lem:size-two atoms LF}
	Let $S$ be a Boolean sublattice whose quarks have size~$2$. If $S$ is a UFS, then the following statements hold.
	\begin{enumerate}
		\item $\mathcal{G}_p(S)$ does not contain any cycle of length $3$.
		\smallskip
		
		\item $\mathcal{G}_p(S)$ does not contain a path/cycle of length $4$.
	\end{enumerate}
\end{lemma}

\begin{proof}
	(1) Observe that the existence of a length-$3$ cycle in $\mathcal{G}_p(S)$ is equivalent to the existence of distinct quarks $\{v_1, v_2\}, \{v_2, v_3\}, \{v_3, v_1\}$ in $S$, in which case $\big\{ \{v_1, v_2\}, \{v_2, v_3\} \big\}$ and $\big\{ \{v_2, v_3\}, \{v_3, v_1\} \big\}$ would be two distinct factorizations of the element $\{v_1, v_2, v_3\}$. However, this is not possible because~$S$ is a UFS.
	\smallskip
	
	(2) First, note that the existence of a length-$4$ path in $\mathcal{G}_p(S)$ is equivalent to the existence of distinct quarks $\{v_1, v_2\}, \{v_2, v_3\}, \{v_3, v_4\}, \{v_4, v_5\}$ in $S$, which would imply that $\big\{ \{v_1, v_2\}, \{v_2, v_3\}, \{v_4, v_5\} \big\}$ and $\big\{ \{v_1, v_2\}, \{v_3, v_4\}, \{v_4, v_5\} \big\}$ are two distinct factorizations of the element $\{v_1, v_2, v_3, v_4, v_5\}$. However, this is not possible because~$S$ is a UFS. Finally, observe that the existence of a length-$4$ cycle is equivalent to the existence of distinct factorizations $\{v_1',v_2'\}, \{v_2', v_3'\}, \{v_3', v_4'\}, \{v_4', v_1'\}$ in $S$, which would imply that $\big\{ \{v_1', v_2'\}, \{v_3',v_4'\} \big\}$ and $\big\{ \{v_1',v_4'\}, \{v_2',v_3'\} \big\}$ are two distinct factorizations of $\{v_1',v_2',v_3',v_4'\}$. This again is not possible because $S$ is a UFS.
\end{proof}

We are in a position to characterize the Boolean sublattices whose quarks have size at most~$2$ that are UFS.

\begin{theorem} \label{thm:UFS characterization for BS with atom-size at most 2}
	Let $S$ be a Boolean sublattice whose quarks have size at most $2$. Then the following conditions are equivalent.
	\begin{enumerate}
		\item[(a)] $S$ is a UFS.
		\smallskip
		
		\item[(b)] For each connected component $C$ of $\mathcal{G}_p(S)$, the Boolean sublattice generated by the edges of $C$ is a UFS.
		\smallskip
		
		\item[(c)] Each connected component of $\mathcal{G}_p(S)$ is a tree with diameter at most $3$.
	\end{enumerate}
\end{theorem}

\begin{proof}
	Since each quark of $S$ that is a singleton must be an isolated quark, in light of Corollary~\ref{cor:isolated atoms UFL} we can assume that $S$ contains no singletons. Set $G := \mathcal{G}_p(S)$.
	\smallskip
		
	(a) $\Rightarrow$ (b): Suppose that $S$ is a UFS. Let $C$ be a connected component of $G$, and let $S_C$ be the Boolean sublattice generated by the set $E(C)$. Since there are no singletons in $S$, a set $A \in S$ is a quark if and only if $|A| = 2$. Therefore $\mathcal{A}(S_C) = E(C) \subseteq \mathcal{A}(S)$. As a result, for each $X \in S_C$, the inclusion $\mathsf{Z}_{S_C}(X) \subseteq \mathsf{Z}_S(X)$ holds (indeed, equality holds). Hence the fact that $S$ is a UFS implies that $S_C$ is also a UFS.

	\smallskip
	(b) $\Rightarrow$ (c): Suppose now that each Boolean sublattice generated by the edges of a connected component of $G$ is a UFS. Let $C$ be a connected component of $G$. If $C$ has exactly one edge, then we are done. Therefore we assume that $|E(C)| \ge 2$.
	
	Suppose first that no edge of $C$ corresponds to an excess quark. Since $|E(C)| \ge 2$, we can pick a vertex $v$ of $C$ such that $\deg v \ge 2$. We claim that every vertex adjacent to $v$ in $C$ has degree $1$. To argue this, fix a vertex $w$ adjacent to $v$ in $C$. Since $\deg v \ge 2$, there exists $A \in \mathcal{A}(S)$ such that $v \in A$ and $A \neq \{v,w\}$. Note that the only vertex adjacent to $w$ in $C$ must be $v$ as otherwise there would exist $B \in \mathcal{A}(S)$ with $B \neq \{v,w\}$ such that $w \in B$ and so $\{v,w\} \subseteq A \cup B$, which is not possible because $\{v,w\}$ is not an excess quark by assumption. Hence each vertex in $C$ adjacent to $v$ has degree~$1$, which implies that $C$ is a star graph, that is, a tree with diameter at most~$2$.
	
	Suppose, on the other hand, that there exists an edge of $C$ corresponding to an excess quark, namely, $\{v,w\}$. Then we can take $v', w' \in V(C) \setminus \{v,w\}$ such that $\{v,v'\}$ and $\{w,w'\}$ both belong to $\mathcal{A}(S)$. Since the Boolean sublattice $\langle E(C) \rangle$ is a UFS, it follows from part~(1) of Lemma~\ref{lem:size-two atoms LF} that $v' \neq w'$. We claim that $\deg v' = 1$. Suppose that this is not the case. Then one could take a vertex $v''$ of $C$ such that $\{v', v''\} \in \mathcal{A}(S)$ and $v'' \neq v$. It follows from part~(1) of Lemma~\ref{lem:size-two atoms LF} that $v'' \neq w$ and, therefore, the edges $\{v'', v'\}, \{v', v\}, \{v,w\}, \{w,w'\}$ would form a length-$4$ path (if $v'' \neq w'$) or a length-$4$ cycle (if $v'' = w'$), which is not possible in light of part~(2) of Lemma~\ref{lem:size-two atoms LF}. Thus, every vertex in $V(C) \setminus \{w\}$ adjacent to $v$ has degree~$1$. Similarly, we can argue that every vertex in $V(C) \setminus \{v\}$ adjacent to $w$ has degree~$1$. Hence $C$ is a tree with diameter~$3$.
	\smallskip
	
	(c) $\Rightarrow$ (a): Finally, suppose that each connected component of $G$ is a tree with diameter at most~$3$. By virtue of Proposition~\ref{prop:UF sufficient condition}, it suffices to check that at most one edge in each connected component of~$G$ corresponds to an excess quark of $S$. Let $C$ be a connected component of $G$. If $C$ has diameter at most~$2$, then $C$ is either a tree with two vertices (if $C$ has diameter~$1$) or a star graph (if $C$ has diameter~$2$), and it is clear that in both cases $C$ has no excess quarks. Now suppose that $C$ has diameter $3$, and let $\{v_0, v_1\}, \{v_1, v_2\}, \{v_2, v_3\}$ be three edges forming a path of length~$3$. Since $C$ has diameter~$3$, every vertex $v$ adjacent to $v_1$ (resp., to $v_2$) with $v \neq v_2$ (resp., $v \neq v_1$) must have degree~$1$. This, along with the fact that no edge of $C$ incident to a degree-$1$ vertex can correspond to an excess quark of $S$, implies that the only edge of $C$ corresponding to an excess quark is $\{v_1, v_2\}$. As a consequence, $S$ is a UFS.
\end{proof}

%Here is a related question.
%
%\begin{question}
%	Can we characterize the geometric lattices that are UFSs?
%\end{question}

\medskip
%\bigskip
%%%%%%%%%%%%%%%%%%%%
%%%%%%%%5%%%%%%%%%%%
\section{The Half-Factorial Property}
\label{subsec:half-factoriality}

In this section, we investigate the property of half-factoriality. In the context of Boolean sublattices, this property is strictly weaker than the unique factorization property. The following simple example illustrates this observation.

\begin{example} \label{ex:HFS-FFS not LFS-UFS}
	Consider the Boolean sublattice $S := \langle 12, 23, 13 \rangle$, which is factorizable with set of quarks $\mathcal{A}(S) = \{ 12, 23, 13 \}$. Although the element $123$ has three factorizations, namely, $\{12, 13\}$, $\{12, 23\}$, and $\{13,23\}$, all of them has length $2$. As a consequence, $S$ is an HFS that is not a UFS.
\end{example}

Our main purpose is to provide a characterization of the Boolean sublattices with quarks of size at most $2$ that are HFSs. We start with a result parallel to Corollary~\ref{cor:isolated atoms UFL}, whose proof immediately follows from part~(4) of Lemma~\ref{lem:decomposition lemma}.

\begin{prop} \label{prop:isolated atoms HFL}
	Let $S$ be a Boolean sublattice, and let $\mathcal{A}_I(S)$ be the set of isolated quarks of~$S$. Then $S$ is an HFS if and only if $\langle \mathcal{A}(S) \setminus \mathcal{A}_I(S) \rangle$ is an HFS.
\end{prop}

Let us now restrict our attention to factorizable Boolean sublattices whose quarks have size at most~$2$. The following lemma will be useful in the proof of Theorem~\ref{thm:HF Boolean sublattices of atom-size 2}.

\begin{lemma} \label{lem:size-two atoms HF}
	Let $S$ be a Boolean sublattice whose quarks have size~$2$. If $S$ is an HFS, then the following statements hold.
	\begin{enumerate}
		\item $\mathcal{G}_p(S)$ does not contain any path of length $5$.
		\smallskip
		
		\item If $\mathcal{G}_p(S)$ has a component $C$ containing a cycle of length~$3$, then $C = C_3$, where $C_3$ is the cycle graph of length~$3$. %does not contain a path/cycle of length $4$.
		\smallskip
		
		\item If $\mathcal{G}_p(S)$ has a component $C$ containing a cycle of length~$5$, then $C = C_5$, where $C_5$ is the cycle graph of length~$5$.
	\end{enumerate}
\end{lemma}

\begin{proof}
	(1) Suppose, for the sake of a contradiction, that $\mathcal{G}_p(S)$ contains a path of length~$5$. Then there exist distinct $v_1, v_2, v_3, v_4, v_5, v_6 \in \nn$ such that $\{v_1, v_2\}, \{v_2, v_3\}, \{v_3, v_4\}, \{v_4, v_5\}, \{v_5, v_6\} \in \mathcal{A}(S)$. In this case, we see that $\{v_1, v_2\} \vee \{v_3, v_4\} \vee \{v_5, v_6\}$ and $\{v_1, v_2\} \vee \{v_2, v_3\} \vee \{v_4, v_5\} \vee \{v_5, v_6\}$ are two factorizations with distinct lengths of the element $\{v_1, v_2, v_3, v_4, v_5, v_6\}$ in $S$, which contradicts the fact that~$S$ is an HFS.
	\smallskip
	
	(2) Now suppose that $C$ is a connected component of $\mathcal{G}_p(S)$ containing a cycle of length~$3$. Then we can take $\{v_1, v_2\}, \{v_2, v_3\}, \{v_3, v_1\} \in \mathcal{A}(S)$ such that $v_1, v_2$, and $v_3$ are distinct vertices of~$C$. Now observe that if $V(C) \setminus \{v_1, v_2, v_3\}$ were nonempty, then we would be able to take $v_4 \in V(C)$ adjacent to some vertex of $C$, say to $v_1$: however, in this case $\{v_1, v_2\} \vee \{v_1, v_3\} \vee \{v_1, v_4\}$ and $\{v_1, v_4\} \vee \{v_2, v_3\}$ would be two factorizations with distinct lengths of the element $\{v_1, v_2, v_3, v_4\}$ in $S$, which is not possible because $S$ is an HFS. Hence $C = C_3$.
	\smallskip
	
	(3) The argument follows the lines of that given to establish part~(2).
\end{proof}

Thus, for a Boolean sublattice to be an HFS, its pairing graph must satisfy the restrictions (1)--(3) in Lemma~\ref{lem:size-two atoms HF}. Observe that trees with diameter at most $4$ satisfy the same restrictions. There is another more general class of connected graphs satisfying such restrictions.

\begin{definition}
	A connected graph $G$ is called a \emph{candy graph} if it satisfies the following two conditions.
	\begin{enumerate}
		\item $G$ has diameter at most~$4$.
		\smallskip
		
		\item $G$ contains a length-$k$ cycle if and only if $k = 4$.
	\end{enumerate}
\end{definition}

Candy graphs are a crucial ingredient for the characterization of half-factoriality we will provide in Theorem~\ref{thm:HF Boolean sublattices of atom-size 2}. Here is an example of a candy graph.

\begin{example} \label{ex:candy graph}
	Let $a, a_1, a_2, a_3, b_1, b_2, b_3, b_4, c, c_1, c_2$ be distinct positive integers, and consider the Boolean sublattice $S$ whose pairing graph is that in Figure~\ref{fig:a candy graph}. It is not hard to verify that $S$ is an HFS.
	\begin{figure}[h]
		\includegraphics[width=6.5cm]{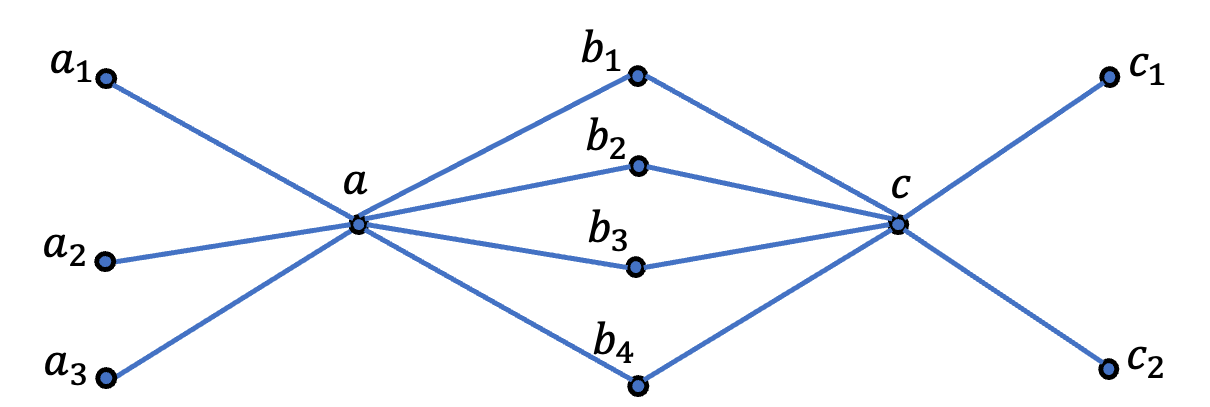}
		\caption{A candy graph with $11$ vertices and $\binom{4}{2}$ cycles.}
		\label{fig:a candy graph}
	\end{figure}
\end{example}

Let us prove some facts about candy graphs. Throughout the proof of Lemma~\ref{lem:candy graphs}, we will choose the notation consistent with that of Example~\ref{ex:candy graph}, and so looking at the candy graph illustrated in such an example may help the reader to follow our argument. 

\begin{lemma} \label{lem:candy graphs}
	For a candy graph $G$, the following statements hold.
	\begin{enumerate}
		\item Every vertex of degree at least $3$ is part of a cycle.
		\smallskip
		
		\item If a non-leaf vertex is adjacent to a vertex of degree at least~$3$, then it has degree~$2$.
		\smallskip
				
		\item $G$ has at most two vertices with degree at least $3$.
		\smallskip
		
		\item If $G$ has two vertices of degree at least $3$, then they are at distance $2$ and their common adjacent vertices  are precisely the vertices of $G$ with degree~$2$.
	\end{enumerate}
\end{lemma}

\begin{proof}
	If $G$ does not contain any vertex of degree at least $3$, then the four statements of the lemma follow trivially. Thus, we assume that $G$ has a vertex $a$ such that $\deg a \ge 3$.
	\smallskip
	
	(1) Suppose, for the sake of a contradiction, that the vertex $a$ is not part of any cycle. Let $\{a, a_1\}, \{a_1, a_2\}, \dots, \{a_{\ell - 1}, a_\ell\}$ be a shortest path from $a$ to a cycle $\{a_\ell, v_1\}, \{v_1, v_2\}, \{v_2, v_3\}, \{v_3, a_\ell \}$. Assume first that $\ell = 1$. Then $a$ is not adjacent to neither $v_1$ nor $v_3$ because $G$ contains no length-$3$ cycles. This, along with the fact that $\deg a \ge 3$, ensures the existence of $a_0 \in V(G) \setminus \{a_1, v_1, v_2, v_3\}$ that is adjacent to $a$. Hence $\{a_0, a\}, \{a, a_1\}, \{a_1, v_1\}, \{v_1, v_2\}, \{v_2, v_3\}$ is a path of length $5$, a contradiction. Assume, therefore, that $\ell \ge 2$. In such a case, $\{a, a_1\}, \{a_1, a_2\}, \dots, \{a_{\ell - 1}, a_\ell\}, \{a_\ell, v_1\}, \{v_1, v_2\}, \{v_2, v_3\}$ is a path of length at least $5$, which contradicts that $G$ is a candy graph.
	\smallskip
	
	(2)  By part~(1), the vertex~$a$ must be part of a length-$4$ cycle, namely, $\{a,b_1\}, \{b_1, c\}, \{c, b_2\}, \{b_2, a\}$. Let $b_1, \dots, b_n$ be the non-leaf vertices that are adjacent to~$a$. Now observe that if $\deg \, b_k \ge 3$ for some $k \in \ldb 1, n \rdb$, then after taking $c' \in V(G) \setminus \{a,c\}$ adjacent to $b_k$ and $b_i, b_j \in \{b_1, \dots, b_n\} \setminus \{b_k\}$, we would be able to construct in $G$ the following length-$5$ path $\{c',b_k\}, \{b_k,a\}, \{a,b_i\}, \{b_i,c\}, \{c, b_j\}$, which is not possible. Hence $\deg b_i = 2$ for every $i \in \ldb 1,n \rdb$.
	\smallskip
	
	(3)-(4) Let the vertices $b_1, \dots, b_n$, and $c$ be defined as in part~(2). We argue first that $b_i$ is adjacent to~$c$ for every $i \in \ldb 1,n \rdb$. Suppose, otherwise, that $b_i$ is not adjacent to $c$ for some $i \in \ldb 3,n \rdb$ (we know that $b_1$ and $b_2$ are adjacent to $c$). Observe that the vertices $b_i$ and $b_j$ are not adjacent for any $i, j \in \ldb 1,n \rdb$ with $i \neq j$ because $G$ contains no length-$3$ cycles. Let $b' \in V(G) \setminus \{a,c\}$ be a vertex adjacent to $b_i$. Since $G$ has no length-$3$ cycles, $b' \notin \{b_1, b_2\}$. As a result, $\{b',b_i\}, \{b_i,a\}, \{a, b_1\}, \{b_1, c\}, \{c, b_2\}$ is a path in $G$ of length~$5$, which contradicts that~$G$ is a candy graph. Hence $b_i$ is adjacent to $c$ for every $i \in \ldb 1,n \rdb$.
	\smallskip
 
	We can mimic the previous argument to prove that every non-leaf vertex adjacent to $c$ is also adjacent to $a$. Thus, the non-leaf vertices adjacent to~$c$ are precisely $b_1, \dots, b_n$. Since $G$ is connected and $\deg b_i = 2$ for every $i \in \ldb 1,n \rdb$,
	\[
		V(G) = A \cup C \cup \big\{ a,c, b_1, \dots, b_n \big\},
	\]
	where $A$ denotes the set of leaves adjacent to $a$ while $C$ denotes the set of leaves adjacent to $c$. Thus, besides $a$, the only vertex of $G$ that can have degree at least $3$ is $c$, which yields statement~(3). Also, if~$c$ is the other vertex of $G$ with degree at least~$3$, then the common adjacent vertices of $a$ and $c$ are $b_1, \dots, b_n$, which are precisely the vertices of $G$ with degree $2$, whence statement~(4) follows.
\end{proof}

We are in a position to characterize half-factoriality inside the class consisting of all Boolean sublattices whose quarks have size at most~$2$.

\begin{theorem} \label{thm:HF Boolean sublattices of atom-size 2}
	Let $S$ be a Boolean sublattice whose quarks have size at most $2$. Then the following conditions are equivalent.
	\begin{enumerate}
		\item[(a)] $S$ is an HFS.
		\smallskip
		
		\item[(b)] For each connected component $C$ of $\mathcal{G}_p(S)$, the Boolean sublattice generated by the edges of $C$ is an HFS.
		\smallskip
		
		\item[(c)] The connected components of $\mathcal{G}_p(S)$ are length-$3$ cycles, length-$5$ cycles, trees with diameters at most $4$, or candy graphs.
	\end{enumerate}
\end{theorem}

\begin{proof}
	Since each quark of $S$ that is a singleton must be an isolated quark, by virtue of Proposition~\ref{prop:isolated atoms HFL} we can assume that every quark of $S$ has size~$2$. Set $G := \mathcal{G}_p(S)$.

	(a) $\Rightarrow$ (b): Let $C$ be a connected component of $G$, and let $S_C$ be the Boolean sublattice generated by $E(C)$. As we have seen in the proof of Theorem~\ref{thm:UFS characterization for BS with atom-size at most 2}, the inclusion $\mathsf{Z}_{S_C}(X) \subseteq \mathsf{Z}_S(X)$ holds for all $X \in C$. Therefore we can deduce that $S_C$ is an HFS from the fact that $S$ is an HFS.
	\smallskip
	
	(b) $\Rightarrow$ (a): For each connected component $C$ of $G$, let $S_C$ be the Boolean sublattice generated by the edges of $C$. Fix a nonempty set $X \in S$. In light of part~(2) of Lemma~\ref{lem:decomposition lemma}, we can write $X = X_1 \sqcup \dots \sqcup X_k$, where $X_i \in S_{C_i}$ for every $i \in \ldb 1,k \rdb$, where $C_1, \dots, C_k$ are distinct connected components of $G$. By part~(3) of Lemma~\ref{lem:decomposition lemma}, the assignments $z = z_1 \sqcup \dots \sqcup z_k\mapsto (z_1, \dots, z_k)$, where $z \in \mathsf{Z}_S(X)$ and $z_i := \{A \in z \mid A \le X_i\}$ for every $i \in \ldb 1,k \rdb$, induce a natural bijection from $\mathsf{Z}_S(X)$ to $\mathsf{Z}_{S_{C_1}}(X_1) \times \dots \times \mathsf{Z}_{S_{C_k}}(X_k)$. Thus, for any two factorizations $z := z_1 \sqcup \dots \sqcup z_k$ and $z' := z'_1 \sqcup \dots \sqcup z'_k$ of~$X$ in $S$, we see that $z_i, z'_i \in \mathsf{Z}_{S_{C_i}}(X_i)$ for every $i \in \ldb 1,k \rdb$, and so the fact that $S_{C_1}, \dots, S_{C_k}$ are HFSs guarantees that $|z| = |z'|$. As a consequence, $S$ is an HFS.
	\smallskip
	
	(b) $\Rightarrow$ (c): Let $C$ be a connected component of $G$. We consider the following cases.
	\smallskip
	
	\textsc{Case 1:} $C$ does not contain any cycle. In this case, $C$ is a tree. In addition, it follows from part~(1) of Lemma~\ref{lem:size-two atoms HF} that the diameter of $C$ is at most $4$.
	\smallskip
	
	\textsc{Case 2:} $C$ contains a cycle. By part~(1) of Lemma~\ref{lem:size-two atoms HF}, the component $C$ must contain a cycle of length $3$, $4$, or $5$. If $C$ contains a cycle of length~$3$, then $C = C_3$ by part~(2) of Lemma~\ref{lem:size-two atoms HF}. Similarly, part~(3) of Lemma~\ref{lem:size-two atoms HF} ensures that $C = C_5$ provided that $C$ contains a cycle of length~$5$. Therefore suppose that $C$ has a cycle $C_4$ with length $4$. The notation of the following three paragraphs is consistent with that of Example~\ref{ex:candy graph}.
	
	Let $\{a,b_1\}, \{b_1,c\}, \{c,b_2\}, \{b_2,a\}$ be the edges of $C_4$. If $C = C_4$, then $C$ is a candy graph, and we are done. Thus, we can assume that one of the vertices of $C_4$, namely $a$, has degree strictly greater than $2$ in~$C$. Let $B$ be the set of non-leaf vertices of $C$ that are adjacent to $a$ (in particular, $b_1, b_2 \in B$). Note that every vertex $b \in B$ can only be adjacent to some of the vertices in $V(C_4)$, as otherwise~$C$ would contain a path of length~$5$, namely, $\{b', b\}, \{b,a\}, \{a,b_1\},\{b_1, c\}, \{c, b_2\}$ in case that $b'$ were a vertex in $V(C) \setminus V(C_4)$ adjacent to $b$; however, this is not possible in light of part (1) of Lemma~\ref{lem:size-two atoms HF}. In addition, each vertex of~$C$ adjacent to the vertex $a$ cannot be adjacent to any vertex in $B$ as, in light of part~(2) of Lemma~\ref{lem:size-two atoms HF}, the graph $C$ cannot contain any $3$-cycles. As the vertices in $B$ are not leaves, the last two observations guarantee that no two vertices in~$B$ are adjacent and also that every vertex in $B$ is adjacent to $c$. Thus, we have proved that every non-leaf vertex of $C$ that is adjacent to~$a$ must be adjacent to $c$. By a completely similar argument, we can check that every non-leaf vertex of~$C$ that is adjacent to $c$ must be adjacent to $a$. Hence $B$ is also the set of non-leaf vertices of $C$ that are adjacent to $c$. Because $C$ is connected, in order to conclude that~$C$ is a candy graph, it suffices to argue that $\deg b = 2$ for every $b \in B$. We consider the following two subcases.
	\smallskip
	
	\textsc{Case 2.1:} $|B| = 2$. Since $\deg a \ge 3$, the equality $|B|=2$ ensures the existence of a leaf $a' \in V(C)$ that is adjacent to $a$. In this case, $\deg b_1 = 2$, as otherwise, one could take a vertex $b' \in V(C) \setminus \{a,b_2,c\}$ adjacent to $b_1$ to create a path with length $5$, namely, $\{b',b_1\}, \{b_1, c\}, \{c, b_2\}, \{b_2, a\}, \{a,a'\}$ (observe that $b' \neq a'$ as $a'$ is a leaf), which is not possible in light of part~(1) of Lemma~\ref{lem:size-two atoms HF}. Similarly, we can argue that $\deg b_2 = 2$. Hence $C$ is a candy graph with precisely one cycle of length~$4$.
	\smallskip
	
	\textsc{Case 2.2:} $|B| \ge 3$. %In this case, we claim that every vertex in $B$ has degree $2$. 
	Fix a vertex $b \in B$. Observe that $\deg b = 2$ as, otherwise, we could take $b' \in V(C) \setminus \{a,c\}$ adjacent to $b$ and then pick distinct $b_i, b_j \in B \setminus \{b\}$ to form the length-$5$ path $\{b',b\}, \{b,a\}, \{a,b_i\}, \{b_i,c\}, \{c, b_j\}$ (observe that $b' \notin \{b_i,b_j\}$ because it is adjacent to $b$), which is not possible in light of part~(1) of Lemma~\ref{lem:size-two atoms HF}. As a result, $C$ is a candy graph (with $\binom{|B|}{2}$ cycles of length~$4$ in case that $B$ is a finite set).
	\smallskip
	
	(c) $\Rightarrow$ (b): We can fix a connected component and assume, without loss of generality, that it is the whole graph $G := \mathcal{G}_p(S)$; that is, we can assume that $G$ is connected. When $G$ is either a length-$3$ or a length-$5$ cycle, we can explicitly check that $S$ is an HFS (but not a UFS).
	\smallskip
	
	Suppose now that $G$ is a tree with diameter at most $4$. Take $X \in S \setminus \{\hat{0}\}$, and then let $G_X$ be the subgraph of $G$ induced by the vertices in~$X$ and let $S_X$ be the Boolean sublattice generated by the edges of $G_X$. It is clear that $\mathsf{Z}_S(X) = \mathsf{Z}_{S_X}(X)$. In light of parts~(2) and~(3) of Lemma~\ref{lem:decomposition lemma}, we can assume that $G_X$ is connected (see the argument used to prove the implication (b) $\Rightarrow$~(a)). Hence $G_X$ is a finite tree with diameter at most $4$. If $G$ has diameter at most $3$, then $S_X$ is a UFS by Theorem~\ref{thm:UFS characterization for BS with atom-size at most 2}. Thus, we assume that $G_X$ has diameter $4$. Let $\{v_1, v_2\}, \{v_2, v_3\}, \{v_3, v_4\}, \{v_4, v_5\}$ be a length-$4$ path in $G_X$, and let $E$ be the set of edges of $G_X$ that are incident to a leaf. It is clear that~$E$ is contained in each factorization of $X$. We claim that $\mathsf{L}_{S_X}(X) = \{|E|\}$ if $v_3$ is adjacent to a leaf and $\mathsf{L}_{S_X}(X) = \{|E| + 1\}$ otherwise. Assume first that $v_3$ is adjacent to a leaf, and let $z$ be a factorization of~$X$ in $S_X$. Note that, as $v_3$ is adjacent to a leaf, $E$ is an edge covering of $G_X$ and, therefore, $z = E$. Hence $\mathsf{L}_{S_X}(X) = \{|E|\}$. Now assume that $v_3$ is not adjacent to any leaf, and let $z'$ be a factorization of $X$ in $S_X$. In this case, there must be a quark $A \in z'$ such that $v_3 \in A$, in which case, $E \cup \{A\}$ is an edge covering of $G_X$, and so $z' = E \cup \{A\}$. As a consequence, $\mathsf{L}_{S_X}(X) = \{|E| +1\}$. Hence $|\mathsf{L}_S(X)| = |\mathsf{L}_{S_X}(X)| = 1$. As $X$ was arbitrarily picked in $S \setminus \{\hat{0}\}$, we conclude that~$S$ is an HFS.
	\smallskip

	Finally, suppose that $G$ is a candy graph. Take $X \in S \setminus \{\hat{0}\}$, and let $G_X$ and $S_X$ be defined as in the previous paragraph. It suffices to show that $\mathsf{L}_{S_X}(X)$ is a singleton. Observe that $G_X$ is either the disjoint union of trees of diameter at most~$4$ or a candy graph. If $G_X$ is the disjoint union of trees of diameter at most~$4$, then parts~ (2) and~(3) of Lemma~\ref{lem:decomposition lemma}, used in tandem with the conclusion of the previous paragraph, guarantee that $S_X$ is an HFS, and so $\mathsf{L}_{S_X}(X)$ is a singleton. Therefore we assume that $G_X$ is a candy graph. Observe that $G_X$ is a finite graph because the set $V(G_X) = X$ is finite. Once again, we adopt a notation consistent with that in Example~\ref{ex:candy graph}. If $G_X$ is a length-$4$ cycle, then we can explicitly verify that $S_X$ is an HFS, and so $\mathsf{L}_{S_X}(X)$ is a singleton. Therefore we can assume that $G_X$ contains a length-$4$ cycle $\{a,b_1\}, \{b_1, c\}, \{c, b_2\}, \{b_2, a\}$ such that $\deg a \ge 3$. Let $a_1, \dots, a_m$ be the leaves that are adjacent to $a$ and let $b_1, \dots, b_n$ be the non-leaf vertices that are adjacent to~$a$. It follows from part~(2) of Lemma~\ref{lem:candy graphs} that $\deg b_i = 2$ for every $i \in \ldb 1,n \rdb$. In addition, we have seen in the proof of parts~(3) and~(4) of Lemma~\ref{lem:candy graphs} that $b_i$ is adjacent to~$c$ for every $i \in \ldb 1,n \rdb$. Now, it follows from parts~(3) and~(4) of Lemma~\ref{lem:candy graphs} that every vertex adjacent to $c$ that is not in $\{b_1, \dots, b_n \}$ must be a leaf. This, along with the fact that $C$ is connected, guarantees that the set of non-leaf vertices of~$C$ is $\{a,c\} \cup \{b_1, \dots, b_n\}$. Let $E$ be the set of edges of $G_X$ that are incident to a leaf. Now suppose that $z$ is a factorization of $X$ in $S_X$. It is clear that $E \subseteq z$. Since $z$ is a factorization of $X$, it is also a minimal edge covering of $G_X$. Therefore, for each $i \in \ldb 1,n \rdb$, precisely one of the two edges incident to $b_i$ is contained in $z$. This implies that $|z| = |E| + n$, and so $\mathsf{L}_{S_X}(X) = \{|E| + n\}$. Since $|\mathsf{L}_S(X)| = |\mathsf{L}_{S_X}(X)| = 1$ for every $X \in S \setminus \{\hat{0}\}$, we conclude that $S$ is an HFS.
\end{proof}

\bigskip
%%%%%%%%%%%%%
%%%%%%%%%%%%%
\section{The Elasticity}
\label{sec:elasticity}

Let $S$ be a factorizable (join) semilattice with $|S| \ge 2$. Borrowing standard notation and terminology from factorization theory, for any $x \in S^\bullet$, we define the \emph{elasticity} $\rho(x)$ of $x$ as follows:
\[
	\rho(x) = \frac{\sup \mathsf{L}(x)}{\min \mathsf{L}(x)}.
\]
Then we call $\rho(S) := \sup\{\rho(x) \mid x \in S^\bullet \}$ the \emph{elasticity} of $S$. Observe that $S$ is an HFS if and only if $\rho(S) = 1$. By definition, $\rho(x) \in \qq_{\ge 1} \cup \{\infty\}$ for all $x \in S^\bullet$. In a similar way, we can define elasticities for other commutative algebraic structures, specially atomic monoids and rings.

Even though the elasticity of every element in a given algebraic structure (e.g., monoid, integral domain, or lattice) is either rational or infinite by definition, the elasticity of the whole structure can be irrational. There are, however, several natural classes of atomic algebraic structures whose members have either rational or infinite elasticities. For instance, the elasticity of any numerical monoid is rational~\cite[Theorem~2.1]{CHM06}, and the elasticity of any ring of integers of an algebraic number field is also rational \cite[Theorem~10]{AACS93}. In addition, the elasticity of any one-dimensional local integral domain is either rational or infinite \cite[Theorem~2.12]{AA92}. It is also known that the elasticity of every (additive) submonoid of $\nn_0^2$ is either rational or infinite \cite[Theorem~5.7]{fG20} (however, it is still unknown whether the same situation happens for all submonoids of any finite-rank free commutative monoid; see, for instance, \cite{sT17} and \cite[Question~5.1]{fG20}).

On the other hand, there are natural classes of atomic algebraic structures whose members can realize any prescribed possible irrational elasticity. Indeed, it was proved in \cite[Theorem~3.2]{AA92} and then in \cite[Proposition~3.5]{CHM06} that for every $\alpha \in \rr_{\ge 1} \cup \{\infty\}$, there exists a Dedekind domain $D$ (with torsion class group) such that $\rho(D) = \alpha$. Also, we can use the formula given in \cite[Theorem~3.2]{GO20} to argue that for each irrational $\alpha \in \rr_{\ge 1} \cup \infty$ there exists an atomic additive submonoid~$M$ of $\qq_{\ge 0}$ with $\rho(M) = \alpha$ (see also \cite[Example~4.3]{GGT21}). It turns out that the class consisting of all Boolean sublattices satisfies a similar property, which we will prove in the next theorem.

\begin{lemma} \label{lem:elasticity aux}
	Let $a_1, \dots, a_k$ and $b_1, \dots, b_k$ be positive integers such that $1 < \frac{a_1}{b_1} \le \dots \le \frac{a_k}{b_k}$. Then
	\[
		\frac{a_1 + \dots + a_k}{b_1 + \dots + b_k} \le \frac{a_k}{b_k}.
	\]
\end{lemma}

\begin{proof}
	Let us prove by induction on $n$ that
	\begin{equation} \label{eq:inequiality for elasticity}
		\frac{a_1 + \dots + a_n}{b_1 + \dots + b_n} \le \frac{a_n}{b_n}
	\end{equation}
for every $n \in \ldb 1,k \rdb$. For $n = 1$, we see that \eqref{eq:inequiality for elasticity} holds trivially. In addition, if \eqref{eq:inequiality for elasticity} holds for some $n < k$, then $\frac{a_1 + \dots + a_n}{b_1 + \dots + b_n} \le \frac{a_{n+1}}{b_{n+1}}$ and, therefore, $(a_1 + \dots + a_{n+1})b_{n+1} \le (b_1 + \dots + b_{n+1}) a_{n+1}$, which means that
\[
	\frac{a_1 + \dots + a_{n+1}}{b_1 + \dots + b_{n+1}} \le \frac{a_{n+1}}{b_{n+1}}.
\]
\end{proof}

We proceed to prove the main result of this section.

\begin{theorem} \label{thm:elasticity of sublattices of B_N}
	For any $\alpha \in \rr_{\ge 1} \cup \{\infty\}$, there exists a Boolean sublattice $S$ with $\rho(S) = \alpha$.
\end{theorem}

\begin{proof}
	If $\alpha = 1$, then it suffices to take $S = B_\nn$; this is because $B_\nn$ is a UFS and so an HFS. Therefore we will assume that $\alpha > 1$. Take a strictly increasing sequence $(q_n)_{n \ge 1}$ of positive rationals such that $\lim_{n \to \infty} q_n = \alpha$, and write $q_n = a_n/b_n$, where $a_n, b_n \in \nn$ and $\gcd(a_n, b_n) = 1$. As $\alpha > 1$, after dropping a finite number of initial terms from the sequence $(q_n)_{n \ge 1}$, we can assume that $1 < q_1$ and also that $\min\{a_n, b_n\} \ge 2$ for every $n \in \nn$. Now let $(J_n)_{n \ge 1}$ be a sequence of pairwise disjoint subsets of $\nn$ satisfying that $|J_n| = a_n b_n$ for every $n \in \nn$. Now for each $n \in \nn$, we create two distinct partitions~$\alpha_n$ and $\beta_n$ of $J_n$ as follows: place the elements of $J_n$ as entries of an $a_n \times b_n$ matrix $M_n$, and let $A_{n,i}$ (resp., $B_{n,j}$) be the set consisting of all the entries in the $i$-th row (resp., $j$-th column) for every $i \in \ldb 1, a_n \rdb$ (resp., $j \in \ldb 1, b_n \rdb$). Hence both $\alpha_n := \{A_{n,i} \mid i \in \ldb 1, a_n \rdb \}$ and $\beta_n := \{B_{n,j} \mid j \in \ldb 1, b_n \rdb \}$ are partitions of $J_n$ with sizes $a_n$ and $b_n$, respectively.
	
	Now consider the Boolean sublattice $S$ generated by the set
	\[
		\mathcal{A} := \big\{ A_{n,i}, B_{n,j} \mid n \in \nn \ \text{and} \ (i,j) \in \ldb 1, a_n \rdb \times \ldb 1, b_n \rdb \big\}.
	\]
	For each $n \in \nn$, the fact that $\alpha_n$ is a partition of $J_n$ implies that $J_n \in S$. Before proving that the elasticity of $S$ is $\alpha$, we need to argue some facts about quarks and factorizations in~$S$.
	\medskip
	
	\noindent \textit{Claim 1.} $S$ is factorizable with $\mathcal{A}(S) = \mathcal{A}$.
	\smallskip
	
	\noindent \textit{Proof of Claim 1.} It suffices to show that $\mathcal{A}(S) = \mathcal{A}$. It is clear that $\mathcal{A}(S) \subseteq \mathcal{A}$. For the reverse inclusion, first note that for each $n \in \nn$, the fact that $\min \{a_n, b_n\} \ge 2$ guarantees that no row (resp., column) of the matrix $M_n$ is contained in a column (resp., row) of $M_n$, and so $A_{n,i}$ and $B_{n,j}$ are not comparable in $S$ for any $(i,j) \in \ldb 1,a_n \rdb \times \ldb 1, b_n \rdb$. On the other hand, for any $m,n \in \nn$ with $m \neq n$ the blocks of the partitions $\alpha_m$ and $\beta_m$ are not comparable in $S$ with the blocks of the partitions $\alpha_n$ and~$\beta_n$ as the sets of entries of the matrices $M_m$ and $M_n$ are disjoint. Hence $\mathcal{A} \subseteq \mathcal{A}(S)$, and so Claim~1 is established.
	\medskip

	\noindent \textit{Claim 2.} For each $n \in \nn$, the following statements hold.
	\begin{enumerate}
		\item $\Lambda_{J_n} = \langle A_{n,i}, B_{n,j} \mid (i,j) \in \ldb 1, a_n \rdb \times \ldb 1,b_n \rdb \rangle$.
		\smallskip
		
		\item $|\mathsf{Z}(J)| = 1$ for every $J \in \Lambda_{J_n} \setminus \{J_n\}$.
		\smallskip
		
		\item $|\mathsf{Z}(J_n)| = 2$ and $\mathsf{L}(J_n) = \{a_n, b_n\}$.
	\end{enumerate}
	\smallskip
	
	\noindent \textit{Proof of Claim 2.} Fix $n \in \nn$.
	\smallskip
	
	(1) Since $A_{n,i} \le J_n$ and $B_{n,j} \le J_n$ for each pair $(i,j) \in \ldb 1, a_n \rdb \times \ldb 1, b_n \rdb$, we see that the inclusion $\langle A_{n,i}, B_{n,j} \mid (i,j) \in \ldb 1, a_n \rdb \times \ldb 1,b_n \rdb \rangle \subseteq \Lambda_{J_n}$ holds. For the reverse inclusion, it suffices to observe that if $J \in S$ with $J \subseteq J_n$, then Claim~1 ensures that $J$ is the union of some of the blocks of~$\alpha_n$ or $\beta_n$.
	\smallskip
	
	(2) Take $J \in \Lambda_{J_n} \setminus \{J_n\}$. As $J \subseteq J_n$, the set $J$ consists of some entries of the matrix $M_n$. Let~$z_1$ and $z_2$ be two factorizations of $J$ in $S$. The quarks appearing in $z_1$ or $z_2$ are some of the blocks of~$\alpha_n$ or $\beta_n$. Suppose that $A_{n,i} \in z_1$ for some $i \in \ldb 1, a_n \rdb$. Then $J$ contains all the entries in the $i$-th row of $M_n$ and so if $A_{n,i} \notin z_2$, then $B_{n,j} \subseteq J$ for all $j \in \ldb 1, b_n \rdb$, which is not possible because $J$ is strictly contained in $J_n$. Hence $A_{n,i} \in z_2$. In a completely similar way, one can check that every quark $A_{n,i}$ contained in $z_2$ must be contained in $z_1$. The quarks of the form $B_{n,j}$ can be treated \emph{mutatis mutandis}, and so we can conclude that $z_1 = z_2$. Thus, $|\mathsf{Z}(J)| = 1$.
	\smallskip
	
	(3) Now suppose that $z$ is a factorization of $J_n$ in $S$ containing the quark $A_{n,i}$ for some $i \in \ldb 1, a_n \rdb$. Suppose, by way of contradiction, that $A_{n,k}$ is not contained in $z$ for some $k \in \ldb 1, a_n \rdb$. Since $J_n$ contains each entry of the matrix $M_n$ and, in particular, each entry of the $k$-th row of $M_n$, the factorization $z$ must contain the quark $B_{n,j}$ for every $j \in \ldb 1, b_n \rdb$. This, along with the fact that $A_{n,i} \le B_{n,1} \vee \dots \vee B_{n, b_n}$ in $S$, contradicts that $z$ is a factorization. Thus, $z = \{A_{n,1}, \dots, A_{n, a_n} \}$. In a similar way, we can argue that if $z'$ is a factorization of $J_n$ containing the quark $B_{n,j}$ for some $j \in \ldb 1, b_n \rdb$, then $z' = \{ B_{n,1}, \dots, B_{n, b_n} \}$. Hence $|\mathsf{Z}(J_n)| = 2$ and $\mathsf{L}(J_n) = \{a_n, b_n\}$, which concludes our argument for Claim~2.
	\medskip

	\noindent \textit{Claim 3.} $J \in S$ has a unique factorization if and only if $J_n \not\subseteq J$ for any $n \in \nn$.
	\smallskip
	
	\noindent \textit{Proof of Claim 3.} For the direct implication, suppose that $J_n \subseteq J$ for some $n \in \nn$. Let $\{A_1, \dots, A_\ell\}$ be a factorization of $J$ and assume, without loss of generality, that there exists $k \in \ldb 1, \ell \rdb$ such that $A_i \le J_n$ if and only if $i \in \ldb 1, k \rdb$. The fact that $J_n \le J$, along with the disjointness of the $J_i$s, guarantees that the join of $\{A_1, \dots, A_k\}$ is $J_n$. As a result, $\{A_{n,1}, \dots, A_{n, a_n} \} \cup \{ A_{k+1}, \dots, A_\ell \}$ and $\{ B_{n,1}, \dots, B_{n, b_n} \} \cup \{ A_{k+1}, \dots, A_\ell \}$ are two distinct factorizations of~$J$.
	
	For the converse, assume that $J_n \not\subseteq J$ for any $n \in \nn$. Suppose that $z := \{A_1, \dots, A_\ell \}$ and $z' := \{A'_1, \dots, A'_m\}$ are two factorizations of $J$ in $S$. Now fix $n \in \nn$, and then set
	\[
		z_n := \{A_i \mid i \in \ldb 1, \ell \rdb \text{ and } A_i \le J_n \} \quad \text{and} \quad z'_n := \{A'_j \mid j \in \ldb 1, m \rdb \text{ and } A'_j \le J_n \}.
	\]
	Set $X = \vee z_n$. Clearly, $X \subseteq J_n$. Observe that $z_n$ and $z'_n$ are two factorizations of $X$ in $S$.  Since $J_n \not\subseteq J$, the set~$X$ is strictly contained in $J_n$. Then $X \in \Lambda_{J_n} \setminus \{J_n\}$, and so $|\mathsf{Z}(X)| = 1$ by part~(2) of Claim~2. Thus, $z_n = z'_n$. After proceeding similarly for every $n \in \nn$, we can conclude $z = z'$. Claim~3 is then established.
	\medskip
	
	\noindent \textit{Claim 4.} For each $X \in S$, there exist unique $n_1, \dots, n_k \in \nn$ with $n_1 < \dots < n_k$ and a unique $J \in S$ disjoint from $J_{n_1} \cup \dots \cup J_{n_k}$ such that $X = J_{n_1} \vee \dots \vee J_{n_k} \vee J$ and $J_n \not\subseteq J$ for any $n \in \nn$.
	\smallskip
	
	\noindent \textit{Proof of Claim 4.} For the existence of such a decomposition, take $n_1, \dots, n_k \in \nn$ with $n_1 < \dots < n_k$ such that $n \in \{n_1, \dots, n_k\}$ if and only if $J_n \subseteq X$, and then set $J := X \setminus (J_{n_1} \cup \dots \cup J_{n_k})$. It is clear that $X = J_{n_1} \vee \dots \vee J_{n_k} \vee J$ and also that $J_n \not\subseteq J$ for any $n \in \nn$. For the uniqueness, first note that because $J_n \not\subseteq J$ for any $n \in \nn$, the equality $\{n_1, \dots, n_k\} = \{ n \in \nn \mid J_n \subseteq J\}$ holds. Therefore $n_1, \dots, n_k$ are uniquely determined. Finally, the disjointness between $J$ and $J_{n_1} \cup \dots \cup J_{n_k}$, in tandem with the equality $X = J_{n_1} \vee \dots \vee J_{n_k} \vee J$, guarantees that $J =  X \setminus (J_{n_1} \cup \dots \cup J_{n_k})$. Hence Claim~4 is established.
	\smallskip
	
 	Finally, let us come back to prove that $\rho(S) = \alpha$. To argue that $\rho(S) \ge \alpha$, it suffices to observe that, by virtue of part~(3) of Claim~2, the equality $\rho(J_n) = q_n$ holds for every $n \in \nn$ and, as a result,
	\[
		\rho(S) = \sup \{\rho(X) \mid X \in S^\bullet\} \ge \sup \{\rho(J_n) \mid n \in \nn \} = \lim_{n \to \infty} q_n = \alpha.
	\]
	Let us argue now the reverse inequality, namely, $\rho(S) \le \alpha$. To do so, fix $X \in S$. By Claim~4, there exist unique $n_1, \dots, n_k \in \nn$ with $n_1 < \dots < n_k$ and a unique $J \in S$ disjoint from $J_{n_1} \cup \dots \cup J_{n_k}$ such that $X = J_{n_1} \vee \dots \vee J_{n_k} \vee J$ and $J_n \not\subseteq J$ for any $n \in \nn$. Claim~3 now guarantees that $J$ has a unique factorization $\omega$ in $S$. Now take $z \in \mathsf{Z}(X)$, and set
	\[
		z_i := \{A \in z \mid A \le J_{n_i}\}
	\]
	for each $i \in \ldb 1,k \rdb$. For each  $i \in \ldb 1,k \rdb$, the equality $\vee \{A \mid A \in z_i\} = J_{n_i}$ means that $z_i$ is a factorization of $J_{n_i}$. Since $J = \vee \{A \in z \mid A \not\leq J_{n_i} \text{ for any } i \in \ldb 1,k \rdb \}$, it follows that
	\[
		\omega = \{A \in z \mid A \not\leq J_{n_i} \text{ for any } i \in \ldb 1,k \rdb \}.
	\]
	As a result, any factorization of $X$ in $S$ has the form $z_1 \cup \dots \cup z_k \cup  \omega$, where $z_i$ is a factorization of~$J_i$ in $S$ for every $i \in \ldb 1,k \rdb$. Therefore, after setting $Y := J_{n_1} \vee \dots \vee J_{n_k}$, we see that
	\[
		\rho(X) = \frac{\sup \mathsf{L}(X)}{\min \mathsf{L}(X)} = \frac{\sup \big\{ |\omega| + \mathsf{L}(Y )\big\} }{\min \big\{ |\omega| + \mathsf{L}(Y) \big\} } = \frac{|\omega| + \sup \big\{\mathsf{L}(Y )\big\} }{|\omega| + \min \big\{ \mathsf{L}(Y) \big\} } \le \rho(Y).
	\]
	In addition, for each $i \in \ldb 1, k \rdb$, since $|z_i| \in \mathsf{L}(J_i)$, part~(3) of Claim~2 ensures that $b_{n_i} \le |z_i| \le a_{n_i}$. As a consequence,
	\[
		\rho(Y) = \frac{\sup \mathsf{L}(Y)}{\min \mathsf{L}(Y)} \le \frac{a_{n_1} + \dots + a_{n_k}}{b_{n_1} + \dots + b_{n_k}} \le q_{n_k},
	\]
	where the last inequality follows from Lemma~\ref{lem:elasticity aux}. As a result, $\rho(X) \le \rho(Y) \le q_{n_k} < \alpha$. Finally, we can take the supremum of the set $\{\rho(X) \mid X \in S^\bullet\}$ to obtain that $\rho(S) \le \alpha$.
\end{proof}

\bigskip
%%%%%%%%%%%%%%%%%%%%%
%%%%%%%%%%%%%%%%%%%%%
\section{The Length-Factorial Property}
\label{sec:length-factoriality}

This subsection is devoted to the length-factorial property. First, we prove that for Boolean sublattices whose quarks have size at most~$2$, being an LFS and being a UFS are equivalent conditions, and then we will exhibit an example of an LFS that is not a UFS. For cancellative commutative monoids, it was proved in~\cite[Proposition~3.1]{BVZ23} that the property of being an LFM implies that of being an FFM. We conclude this paper constructing an LFS that is not an FFS.
\smallskip

Recall that the set of isolated quarks of a Boolean sublattice $S$ is denoted by $\mathcal{A}_I(S)$. It follows directly from part~(4) of Lemma~\ref{lem:decomposition lemma} that a Boolean sublattice is an LFS if and only if the sublattice generated by its set of non-isolated quarks is an LFS. We record this observation as the following proposition.

\begin{prop} \label{prop:isolated atoms LFL}
	A Boolean sublattice $S$ is an LFS if and only if $\langle \mathcal{A}(S) \setminus \mathcal{A}_I(S) \rangle$ is an LFS.
\end{prop}

In the class consisting of integral domains, it is well-known that satisfying the length-factorial property is equivalent to being a UFD (see \cite[Corollary~2.11]{CS11}). However, in the class consisting of all Boolean sublattices, being an LFS does not imply being a UFS. The following example sheds some light upon this observation.

\begin{example} \label{ex:LFL not UFL}
	Consider the Boolean sublattice $S$ generated by the set $\{ 123, 456, 14, 25, 36 \}$. It is clear that $\mathcal{A}(S) = \{ 123, 456, 14, 25, 36 \}$. From the Hasse diagram of $S$, which is illustrated in Figure~\ref{fig:LFS not UFS}, we can verify that every element of $S \setminus \{\hat{0}\}$ has a unique factorization except the element $123456$, which has exactly two factorizations, namely, $z_1 := \{ 123, 456 \}$ and $z_2 := \{ 14, 25, 36 \}$. Since $|z_1| \neq |z_2|$, we conclude that~$S$ is an LFS that is not a UFS.
	\begin{figure}[h]
		\includegraphics[width=14cm]{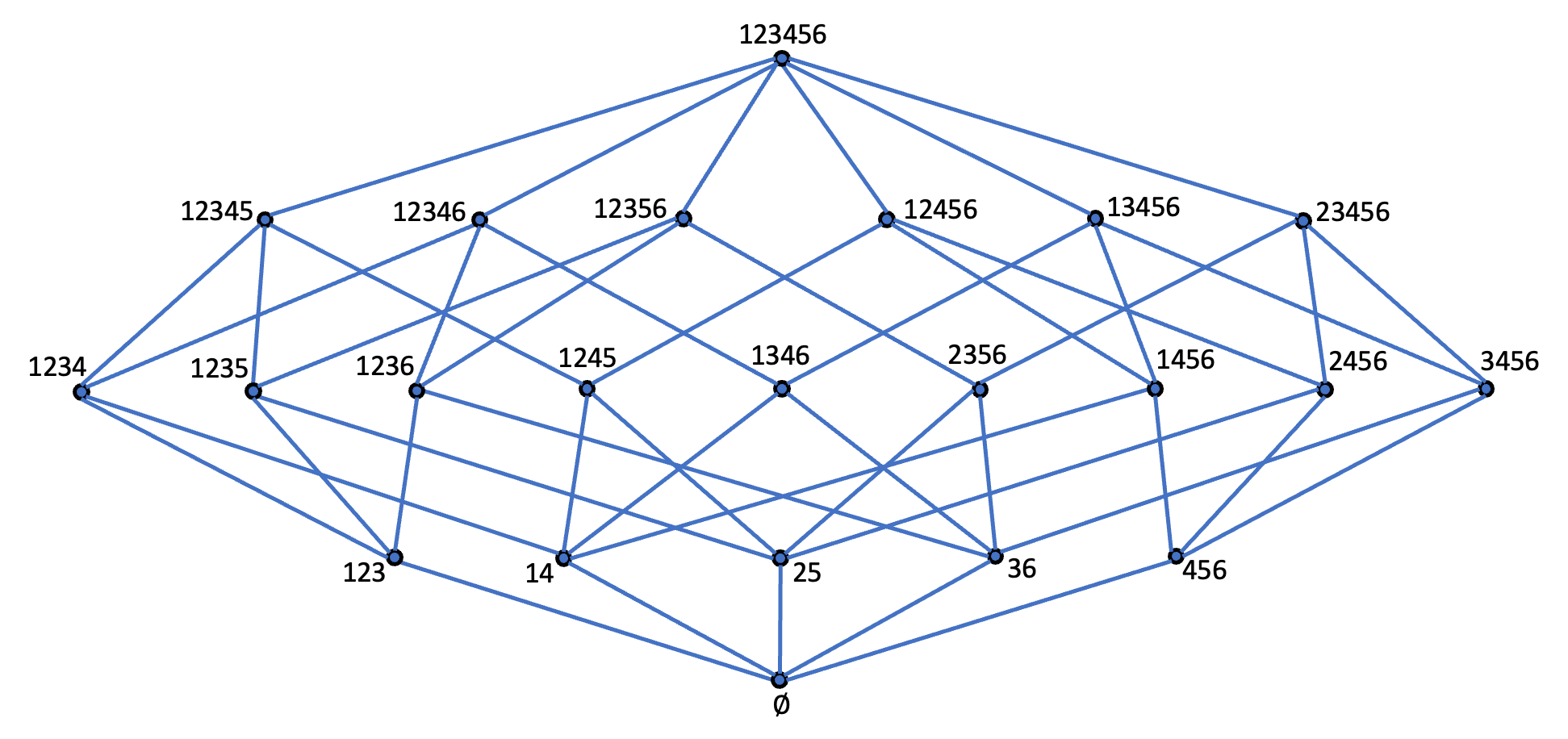}
		\caption{Hasse diagram of the Boolean sublattice generated by the set $\{ 123, 456, 14, 25, 36 \}$.}
		\label{fig:LFS not UFS}
	\end{figure}
\end{example}

Although we have seen in Example~\ref{ex:LFL not UFL} that the property of being an LFS and that of being a UFS are not equivalent in the whole class of Boolean sublattices, they are equivalent in the subclass consisting of all Boolean sublattices whose quarks have size at most~$2$. We present this result in the next proposition.

\begin{prop}
	A Boolean sublattice whose quarks have size at most~$2$ is an LFS if and only if it is a UFS.
\end{prop}

\begin{proof}
	It suffices to argue that each Boolean sublattice whose quarks have size at most~$2$ that is an LFS is also a UFS. Let $S$ be a Boolean sublattice whose quarks have size at most~$2$, and assume that~$S$ is an LFS. By virtue of Proposition~\ref{prop:isolated atoms LFL}, we can assume that every quark of $S$ has size~$2$. %Let~$C$ be a connected component of $\mathcal{G}_p(L)$. 
	Mimicking the proof of Lemma~\ref{lem:size-two atoms LF}, we can arrive to the conclusion that $\mathcal{G}_p(S)$ does not contain length-$3$ cycles, length-$4$ cycles, and length-$4$ paths (this is because the proof of Lemma~\ref{lem:size-two atoms LF} is based on the existence of distinct factorizations with same length). Hence each connected component of $\mathcal{G}_p(S)$ must be a tree of diameter at most~$3$. Now the argument given in the last paragraph of the proof of Theorem~\ref{thm:UFS characterization for BS with atom-size at most 2} shows that $S$ is a UFS.
\end{proof}

Let us conclude comparing the property of being length-factorial with that of having finite factorizations in the context of join semilattices. As an immediate consequence of Remark~\ref{rem:atomic = FF in Boolean sublattices}, we obtain that every Boolean sublattice that is an LFS is also an FFS (as mentioned at the beginning of this section, the corresponding statement holds in the class of cancellative commutative monoids). Observe, on the other hand, that each non-LFS factorizable Boolean sublattice is an example of an FFS that is not an LFS. However, there are join semilattices that are LFSs but not FFSs. The following example confirms our assertion.

\begin{example} \label{ex:LFL not FFL}
	For $i,j \in \nn$ with $i \le j$, we define the function $\beta_{i,j} \colon [0,1) \to \nn_0$ as follows: given $x \in [0,1)$, take the sequence of digits from the $i$-th position to the $j$-th position (both included) after the decimal point of the binary representation of~$x$, and then let $\beta_{i,j}(x)$ be the positive integer represented by the obtained binary string. For example,
	\begin{itemize}
		\item $\beta_{4,6}(1/3) = \beta_{4,6}(0.010101 \ldots_2) = 101_2 = 5$,
		\smallskip
		
		\item $\beta_{2,5}(1/4) = \beta_{2,5}(0.01000_2) = 1000_2 = 8$, and
		\smallskip
		
		\item $\beta_{2,2}(1/4) = \beta_{2,2}(0.01000_2) = 1_2 = 1$.
	\end{itemize}
	For each $n \in \nn$, let $T_n$ denote the $n$-th triangular number $ \frac{n(n+1)}{2}$. Since $T_n - T_{n-1} = n$, the image of the function $\beta_{T_{n-1} + 1, T_n}$ is $\ldb 0, 2^n-1 \rdb$. Now, for $n \in \nn$ and $k \in \ldb 0, 2^n-1 \rdb$, let $A^n_k$ be the preimage of~$k$ by $\beta_{T_{n-1}+1, T_n}$; that is,
	\[
		A^n_k = \big\{ x \in [0,1) \mid \beta_{T_{n-1}+1,T_n}(x) = k \big\}.
	\]
	For example, as $T_2 + 1 = 4$ and $T_3 = 6$, the fact that $\beta_{4,6}(1/3) = 5$ implies that $1/3 \in A^3_5$. For each~$n \in \nn$, it is clear that $\{ A^n_k \mid k \in \ldb0,2^n-1 \rdb \}$ is a partition of $[0,1)$, which we call in the scope of this example the $n$-\emph{th layer}. Now, we let $S$ denote the subsemilattice of $2^{[0,1)}$ (ordered under set inclusion) that is generated by all the sets $A^n_k$, where $n \in \nn$ and $k \in \ldb 0,2^n-1 \rdb$; that is,
	\[
		S = \big\langle A^n_k \mid n \in \nn \text{ and } k \in \ldb 0, 2^n-1 \rdb \big\rangle.
	\]
	Since every layer is a (finite) partition of $[0,1)$, it follows that $[0,1) \in S$. We will prove that $S$ is an LFS that is not an FFS. To accomplish this, it is enough to argue that $S$ is a factorizable join semilattice satisfying that $[0,1)$ has infinitely many factorizations, each of them having a different length, and also that each of the remaining elements of $S$ has a unique factorization.
	\smallskip
	
	\noindent \textit{Claim 1.} $S$ is factorizable with $\mathcal{A}(S) = \{ A^n_k \mid n \in \nn \text{ and } k \in \ldb 0, 2^n-1 \rdb \}$.
	\smallskip
	
	\noindent \textit{Proof of Claim 1.} It suffices to show that there are no inclusion relations between members in the defining generating set $\{ A^n_k \mid n \in \nn \text{ and } k \in \ldb 0, 2^n-1 \rdb \}$. Consider the distinct blocks $A^m_i$ and $A^n_j$, where $m,n \in \nn$ and $(i,j) \in \ldb 0, 2^m-1 \rdb \times \ldb 0, 2^n - 1 \rdb$. As distinct blocks in the same layer are disjoint, we can assume that $m \neq n$. As $m \neq n$, the intervals $\ldb T_{m-1}+1, T_m \rdb$ and $\ldb T_{n-1}+1, T_n \rdb$ are disjoint, which implies that the functions $\beta_{T_{m-1}+1, T_m}$ and $\beta_{T_{n-1}+1, T_n}$ depend on disjoint ranges of digits, which in turn implies that both sets $A^m_i \setminus A^n_j$ and $A^n_j \setminus A^m_i$ are nonempty. Thus, none of the sets $A^m_i$ and $A^n_j$ is included in the other one, and so Claim~1 follows.
	\smallskip
	
	Since each layer is a partition of $[0,1)$, it follows as an immediate consequence of Claim~1 that the $n$-th layer is a factorizable of $[0,1)$ of length $2^n$. In particular, $S$ is not an FFS. To argue that $S$ is an LFS we use the following technical claim.
	\smallskip
	
	\noindent \textit{Claim 2.} The following statements hold.
	\begin{enumerate}
		\item If $X \in S \setminus \{[0,1)\}$, then for each $n \in \nn$ and $k \in \ldb 0, 2^n - 1 \rdb$, the inclusion $A^n_k \subseteq X$ implies that $A^n_k \in z$ for any $z \in \mathsf{Z}(X)$.
		\smallskip
		
		\item $\mathsf{Z}([0,1)) = \big\{ \{ A^n_k \mid k \in \ldb 0, 2^n-1\rdb \} \mid n \in \nn \big\}$.
	\end{enumerate}
	\smallskip
	
	\noindent \textit{Proof of Claim 2.} (1) Suppose that $A^n_k \subseteq X$ for some $n \in \nn$ and $k \in \ldb 0, 2^n-1 \rdb$, and assume, by way of contradiction, that $A^n_k \notin z$. As $X \neq [0,1)$, for each $m \in \nn$, there exists $k_m \in \ldb 0, 2^m - 1 \rdb$ such that $A^m_{k_m} \notin z$. After replacing $A^n_{k_n}$ by $A^n_k$, we can assume that $k_n = k$. Let $x^* \in [0,1)$ be the real number such that $\beta_{T_{m-1}+1, T_m}(x^*) = k_m$ for every $m \in \nn$ (there is a unique $x^* \in [0,1)$ satisfying this condition due to the fact that $\{\ldb T_{m-1}+1, T_m \rdb \mid m \in \nn \}$ is a partition of $\nn$). Now fix $A^\ell_j \in z$. As $A^\ell_{k_\ell} \notin z$, the fact that $A^\ell_j$ and $A^\ell_{k_\ell}$ are blocks in the same layer ensures that $A^\ell_j \cap A^\ell_{k_\ell}$ is empty. Thus, the fact that $x^* \in A^\ell_{k_\ell}$ implies that $x^* \notin A^\ell_j$. Hence $x^* \notin A$ for any $A \in z$, and we conclude that $x^* \notin \vee z = X$. However, as $A^n_k \subseteq X$, the previous conclusion contradicts that $\beta_{T_{n-1}+1, T_n}(x^*) = k$, and so part~(1) of Claim~2 follows.
	\smallskip
	
	(2) If for some $n \in \nn$, a factorization $z \in \mathsf{Z}([0,1))$ contains all the blocks of the $n$-th layer, then it is clear that $z = \{A^n_k \mid k \in \ldb 0, 2^n-1 \rdb \}$. Suppose, by way of contradiction, that $z \in \mathsf{Z}([0,1))$ does not contain all the blocks of any layer. Then, as in the previous part, for each $m \in \nn$, there exists $k_m \in \ldb 0, 2^m - 1 \rdb$ such that $A^m_{k_m} \notin z$, and we can define $x^* \in [0,1)$ similarly to arrive to the desired contradiction.
	\smallskip
	
	By virtue of part~(2) of Claim~2, any two distinct factorizations of $[0,1)$ in~$S$ have different lengths. Thus, proving that $S$ is an LFS amounts to verifying that each element $X \in S \setminus \{\emptyset, [0, 1) \}$ has a unique factorization in~$S$. Fix $z \in \mathsf{Z}(X)$. If $z' \in \mathsf{Z}(X)$ and $A \in z'$, then $A \subseteq X$, and so it follows from part~(1) of Claim~2 that $A \in z$. Therefore $z' \subseteq z$, which implies that $z' = z$. Hence $X$ has a unique factorization. As a consequence, we can now conclude that~$S$ is an LFS that is not an FFS.
\end{example}

\bigskip
%%%%%%%%%%%%%%%
%%%%%%%%%%%%%%%
\section*{Acknowledgments}

While working on this paper, the second author was kindly supported by the NSF awards DMS-1903069 and DMS-2213323.

\bigskip
%%%%%%%%%%%%%%%
%%%%%%%%%%%%%%%

\end{document}